\documentclass[a4paper,12pt,oneside]{article} 
\usepackage{amsfonts,amssymb}
\usepackage{color} \usepackage{mathrsfs} \let\mathcal\mathscr
\usepackage[all,ps,cmtip]{xy}

\title{\sc Categorical crepant resolutions of singularities and the Tits-Freudenthal magic square}
\author{\sc Roland Abuaf \footnote{Institut Fourier, 100 rue des maths,
38402, Saint
Martin d'H\`eres, France. E-mail :\it{abuaf@ujf-grenoble.fr}}}

\usepackage[dvips]{graphicx}

\usepackage{pstricks}
\usepackage{amsfonts,amssymb}
\usepackage{color} \usepackage{mathrsfs} \let\mathcal\mathscr

\usepackage[all,ps,cmtip]{xy}

\DeclareGraphicsExtensions{eps}
\DeclareGraphicsExtensions{pdf}

\usepackage[T1]{fontenc}
\usepackage{amssymb}

\usepackage{amsmath}
\usepackage{latexsym}

\usepackage[latin1]{inputenc}

\newtheorem{theo}{Theorem}[subsection]

\newtheorem{exem}[theo]{Example}
\newtheorem{rem}[theo]{Remark}
\newtheorem{prop}[theo]{Proposition}
\newtheorem{quest}[theo]{Question}

\newtheorem{defi}[theo]{Definition}
\newtheorem{nota}[theo]{Notations}
\newtheorem{lem}[theo]{Lemma}
\newtheorem{cor}[theo]{Corollary}
\newtheorem{conj}[theo]{Conjecture}

\def\DB{\mathrm{D^{b}}}
\def\DM{\mathrm{D^{-}}}
\def\DP{\mathrm{D^{perf}}}

\def\Ri{\mathrm{R^{i}}}
\def\R0{\mathrm{R^{0}}}
\def\ot{\otimes}

\def\HH{\mathrm{Hom}}
\def\Hh{\mathcal{H}om}
\def\LL{\mathrm{\textbf{L}}}
\def\Li{\mathrm{L^{i}}}
\def\RR{\mathrm{\textbf{R}}}
\def\OO{\mathcal{O}}
\def\D{\mathcal{D}}
\def\pt{{\pi_{\mathcal{T}}}_*}
\def\d{\delta}
\def\w{\omega}

\def\X{\tilde{X}}
\def\C{\mathcal{C}}

\def\>{\rightarrow}
\def\F{\mathcal{F}}
\def\T{\mathcal{T}}
\def\tg{\tau(\mathrm{G}_{\omega}(\mathbb{A}^3,\mathbb{A}^6))}
\def\A{\mathcal{A}}
\def\B{\mathcal{B}}

\def\H{\mathcal{H}}
\def\E{\mathcal{E}}
\def\AA{\mathbb{A}}

\def\ma{m_{\mathbb{A}}}
\def\GG{\mathrm{G}_{\omega}(\mathbb{A}^3,\mathbb{A}^6)}

\def\AP{\mathbb{A}\mathbb{P}^2}
\def\Q{\mathcal{Q}}
\def\e12{E_1^{(2)}}
\def\p12{\pi_{{\T_1},{\T_2}}}
\def\d{\delta}

\def\F{\mathcal{F}}
\def\T{\mathcal{T}}
\def\tg{\tau(\mathrm{G}_{\omega}(\mathbb{A}^3,\mathbb{A}^6))}
\def\A{\mathcal{A}}
\def\B{\mathcal{B}}
\def\a{\alpha}
\def\b{\beta}
\def\H{\mathcal{H}}
\def\AA{\mathbb{A}}
\def\G{\mathrm{G}}
\def\ma{m_{\mathbb{A}}}
\def\GG{\mathrm{G}_{\omega}(\mathbb{A}^3,\mathbb{A}^6)}

\def\AP{\mathbb{A}\mathbb{P}^2}
\def\Q{\mathcal{Q}}
\newcommand{\leftexp}[2]{{\vphantom{#2}}^{#1}{#2}}

\newenvironment{proof}
{
\noindent
\textit{\underline{Proof}} :\\
$\blacktriangleright\;$%
}
{\hspace{\stretch{1}}%
$\blacktriangleleft$}

{
\noindent
\textit{\underline{Proof of the main theorem}} :\\
$\blacktriangleright\;$%
}
{\hspace{\stretch{1}}%
$\blacktriangleleft$}

{
\noindent
\textit{\underline{Proof of theorem 3.2.4}} :\\
$\blacktriangleright\;$%
}
{\hspace{\stretch{1}}%
$\blacktriangleleft$}

{
\noindent
\textit{\underline{Sketch of the Proof}} :\\
$\blacktriangleright\;$%
}
{\hspace{\stretch{1}}%
$\blacktriangleleft$}

{
\noindent
\textit{\underline{Preuve}} : (id\'ee)\\
$\blacktriangleright\;$%
}
{\hspace{\stretch{1}}%
$\blacktriangleleft$}

\begin{document}

\maketitle

\begin{abstract}
We prove that the tangent developables of the varieties appearing in the third
row of the Tits-Freudenthal magic square admit categorical crepant resolutions
of singularities.

\end{abstract}

\vspace{\stretch{1}}

\newpage

\begin{section}{Introduction}
We work over $\mathbb{C}$ the field of complex numbers. If $X$ is an algebraic
scheme of finite typer over $\mathbb{C}$, we denote by $\DB(X)$ (resp. $\DM(X)$,
$\DP(X)$), the derived category of bounded complexes of coherent sheaves on $X$
(resp. derived category of unbounded complexes from below of coherent sheaves on
$X$, the full subcategory of $\DB(X)$ consisting of complexes of vector
bundles).

\begin{subsection}{Categorical crepant resolution of singularities}
Let $X$ be an algebraic variety with Gorenstein singularities. A crepant
resolution of singularities of $X$ (that is a resolution $\pi : \tilde{X}
\rightarrow X$ such that $\pi^* \omega_X = \omega_{\tilde{X}}$) is often
considered to be a "minimal" resolution of $X$. The following conjecture (see
\cite{bo}) gives a precise meaning to that notion of minimality:

\begin{conj}[Bondal-Orlov] Let $X$ be a variety with Gorenstein and rational
singularities and let $\tilde{X} \rightarrow X$ be a crepant resolution of $X$.
Then, for any other resolution $\tilde{X}' \rightarrow X$, there exists a fully
faithful embedding:
\begin{equation*}
\DB(\tilde{X}) \hookrightarrow \DB(\tilde{X}').
\end{equation*} 
\end{conj}
Unfortunately, crepant resolution of singularities are quite rare. For instance,
a cone over $v_2(\mathbb{P}^n) \subset \mathbb{P}^{\frac{n(n+1)}{2}}$ never
admits a crepant resolution of singularities when $n$ is odd (it is
$\mathbb{Q}$-factorial with terminal singularities). Thus it seems interesting
to look for "categorical crepant resolution of singularities".

\bigskip

The notion of categorical crepant resolution of singularities has been
formalized by Kuznetsov (see \cite{kuz}) in the case of Gorenstein varieties
with rational singularities.

\begin{defi}
Let $X$ be an algebraic variety with Gorenstein and rational singularities. A
\emph{categorical resolution of singularities} of $X$ is a
triangulated category $\T$ with a functor $\RR {\pi_{\T}}_* : \T \rightarrow
\DB(X)$ such that:
\begin{itemize}
\item there exists a resolution of singularities $\pi : \tilde{X} \rightarrow
X$ with a fully faithful admissible functor $\d : \T \hookrightarrow
\DB(\tilde{X})$ such that $\RR
{\pi_{\T}}_* = \RR \pi_* \circ \d$,

\item for all $\F \in \DP(X)$, we have:

\begin{equation*}
\RR {\pi_{\T}}_* \LL \pi_{\T}^* \F \simeq \F,
\end{equation*}
where $\LL \pi_{\T}^*$ is the left adjoint to $\RR {\pi_{\T}}_*$.
\end{itemize}
\bigskip

\noindent Moreover, if for all $T \in \T$ we have:
\begin{equation*}
 \d (S_{\T}(F)) = \d T \ot \pi^* \w_X [\dim X], 
\end{equation*}
where $S_{\T}$ is the Serre functor of $\T$, we say that $\T$ is \emph{strongly
crepant}.

\noindent If for all $\F \in \DP(X)$, there is a quasi-isomorphism:
\begin{equation*}
\LL \pi_{\T}^* \F \simeq \LL \pi_{\T}^! \F,
\end{equation*}
where $\LL \pi_{\T}^!$ is the right adjoint of $\RR {\pi_{\T}}_*$, we say that
$\T$ is \emph{weakly crepant}.

\end{defi}

Obviously, if $\T \rightarrow \DB(X)$ is a strongly crepant resolution, then it
is weakly crepant. The converse is false as shown in section $7$ and $8$ of
\cite{kuz}. If $\pi : \tilde{X} \rightarrow X$ is a crepant
resolution of singularities  then $\RR \pi_* : \DB(\tilde{X}) \rightarrow
\DB(X)$ is a strongly crepant categorical resolution of singularities. The
converse is partially true:

\begin{prop}
Let $X$ be a projective irreducible Gorenstein variety with rational
singularities. Let $\pi : \X \rightarrow X$ be a proper morphism with $\X$
irreducible, such that $\RR \pi_* : \DB(\X) \rightarrow \DB(X)$ is a weakly
crepant categorical resolution of singularities. Then $\pi : \X \rightarrow X$
is a crepant resolution of singularities.
\end{prop}

\begin{proof}
As $\DB(\X)$ is a weakly crepant categorical resolution of $X$, we have the
equality:

\begin{equation*}
\RR \pi_* \LL \pi^* \mathbb{C}(x) \simeq \mathbb{C}(x),
\end{equation*}
for all $x \in X_{smooth}$, which implies that $\pi$ is dominant. As it is
proper, it is surjective.
By hypothesis, $\DB(\X)$ is an admissible subcategory of the derived category of
a smooth projective variety. It implies that $\DB(\X)$ is
$\mathrm{Ext}$-bounded, so that
$\X$ is smooth. Moreover, we deduce that the right adjoint to $\RR \pi_*$
satisfies the formula (see \cite{neeman}):
\begin{equation*}
\LL \pi^! \F \simeq \LL \pi^* \F \ot \omega_{\X/X}[\dim \X - \dim X].
\end{equation*}
Since $\DB(\X)$ is a weakly crepant categorical resolution of singularities, we
have $\dim \X = \dim X$ and $\omega_{\X} = \pi^* \omega_X$. But the morphism
$\pi$ is surjective, so that the equality $\dim \X = \dim X$ implies that $\pi$
is generically finite.

Using again the fact that $\DB(\X)$ is a weakly crepant categorical resolution,
we have $\RR \pi_* \OO_{\X} = \OO_X$. As $\pi$ is proper, generically finite
and $X$ is normal, Zariski's Main Theorem implies that $\pi$ is birational .

\end{proof}

\end{subsection}

\begin{subsection}{Main result and connections with other works}

\begin{nota}
From now on, we will exclusively focus on \emph{weakly crepant categorical
resolution of singularities}. We will simply call them \emph{categorical crepant
resolution of singularities}.
\end{nota}

The main result of this chapter is the following:

\begin{theo}
The tangent developables of the following embedded varieties admit categorical
crepant resolutions of singularities:
\begin{itemize}
\item The symplectic Grassmannian $\mathrm{G_{\omega}(3,6)} \subset
\mathbb{P}^{13}$,
\item The Grassmannian $\mathrm{G(3,6)} \subset \mathbb{P}^{19},$
\item The spinor variety $\mathbb{S}_{12} \subset \mathbb{P}^{31},$
\item The octonionic Grassmannian: $\mathrm{G_{\w}(\mathbb{O}^3, \mathbb{O}^6)}
\subset \mathbb{P}^{55}.$
\end{itemize}
\end{theo}
These four varieties have a uniform description in terms of complex composition
algebras (this will be discussed in section $2$). They are the "symplectic
Grassmannians" of $\mathbb{A}^3 \subset \mathbb{A}^6$ for $\AA$ the
complexification of $\mathbb{R}, \mathbb{C}, \mathbb{H}$ or $\mathbb{O}$ and
they appear as the varieties in the third row of the Freudenthal's magic square
(see \cite {manilands}). 

\bigskip

In \cite{abuafcategorical}, we define the notion of \emph{wonderful
resolution of singularities} (see definition $2.1.2$) and we
prove the following (see theorem $2.3.2$):

\begin{theo}
Let $X$ be a Gorenstein variety with rational singularities. Assume that $X$
admits a wonderful resolution of singularities, then $X$ admits a categorical
crepant resolution of singularities.
\end{theo}
As a corollary of this result, we obtain (see example $2.1.3$ in
\cite{abuafcategorical}):

\begin{cor} 
All Gorenstein determinantal varieties (square, symmetric or skew-symmetric)
admit categorical crepant resolutions of singularities.
\end{cor}
In Example $2.1.6$ of \cite{abuafcategorical}, we observed that the tangent
developable of
$\mathrm{G}(3,6)$ does not admit a
wonderful resolution of singularities. So the construction of a categorical
crepant resolution of singularities for the tangent variety of $\mathrm{G}(3,6)$
was still an open question. We solve this problem in the present chapter. Note
that the construction of such a categorical resolution of singularities could
also be useful for the (still
conjectural) determination of a homological projective dual to $\mathrm{G}(3,6)$
(see \cite{deliu}).

\bigskip

\noindent \textbf{Acknowledgements} I would like to thank Laurent Manivel for
his invaluable help during the preparation of this paper. I am also very
grateful to Sasha Kuznetsov who suggested to me 
the road map to answer the last
question in the conclusion of the present paper. Finally, stimulating
discussions with Laurent Gruson and Christian Peskine helped me a lot in the
writing of the appendix A.

\end{subsection}
\end{section}

\section{Resolution of singularities and the Tits-Freudenthal magic square}

\subsection{Basic description of the magic square}
One incarnation of the Tits-Freudenthal magic square is a table of $16$
varieties which are linked to each other
by very interesting geometric and representation-theoretic properties (see
\cite{manilands} for a detailed study of the magic square):

\begin{center}
\begin{tabular}{cccc} 
$v_2(Q_1)$ & $\mathbb{P}(T_{\mathbb{P}^2})$ & $\mathrm{G}_{\omega}(2,6)$ &
$\mathbb{OP}^2_0$ \\ 
$v_2(\mathbb{P}^2)$ & $\mathbb{P}^2 \times \mathbb{P}^2$ & $\mathrm{G}(2,6)$ &
$\mathbb{OP}^2$\\
$\G_{\omega}(3,6)$ & $\G(3,6)$ & $\mathbb{S}_{12}$ & $\mathrm{E_7/P_7}$ \\
$\mathrm{F_4^{ad}}$ & $\mathrm{E_6^{ad}}$ & $\mathrm{E_7^{ad}} $&$
\mathrm{E_8^{ad}}$
\end{tabular} 
\end{center}

The second row of this table enumerates the Severi varieties. Recall that a
Severi variety is a smooth variety $X \subset \mathbb{P}^N$ such that
$\frac{3}{2}\dim X +2 = N$ and the secant variety of $X$ does not fill
$\mathbb{P}^N$ (see \cite{zak} for the classification of the
Severi varieties). The varieties in the first row are hyperplane sections of the
Severi varieties. The ones in the last row are the closed orbits of the adjoint
representations of the exceptional groups $\mathrm{F_4, E_6, E_7}$ and
$\mathrm{E_8}$, while the third row gives the varieties of lines through a point
of the corresponding adjoint varieties.

One can also describe these varieties in terms of complex composition algebras.
Let $\AA$ denotes the complexification of one of the four real division algebra
($\mathbb{R}, \mathbb{C}$, $\mathbb{H}$ and $\mathbb{O}$). Let $W_{\AA}$ be the
space of $3 \times 3$
Hermitian matrices over $\AA$. The varieties of the second row can be seen as
the varieties of matrices of rank $1$ in $\mathbb{P}(W_{\AA})$, thus they are
Veronese embeddings of the projective planes over $\AA$ (we will denote them by
$\AP$). The varieties in the
first row are the traceless matrices of rank $1$ in $\mathbb{P}(W_{\AA})$ : they
are hyperplane sections of the previous ones. The varieties in the third row can
be described as $\G_{\omega}(\AA^3, \AA^6)$, the isotropic Grassmannians of
${\AA}^3$ in ${\AA}^6$. The varieties in the last row are the so-called
\textit{F-symplecta}, which we denote by $E(\AA)^{ad}$. We refer to
\cite{manilands} for more details on the
description of the magic square in terms of complex composition algebras. In the
following, we let $\ma = \dim_{\mathbb{C}} \AA$ and $\mathrm{Sp_6}(\AA)$ denotes
the groups:
$\mathrm{Sp_6}$, $\mathrm{SL_6}$, $\mathrm{Spin_{12}}$ and
$\mathrm{E_7}$. We summarize our notations:

\begin{center}
\begin{tabular}{c|cccc}
$\AA$ & $\mathbb{R} \ot_{\mathbb{R}} \mathbb{C}$ & $\mathbb{C} \ot_{\mathbb{R}}
\mathbb{C}$ & $\mathbb{H} \ot_{\mathbb{R}} \mathbb{C}$ & $\mathbb{O}
\ot_{\mathbb{R}} \mathbb{C}$ \\
\hline

$\AP_{0}$ & $v_2(Q_1)$ & $\mathbb{P}(T_{\mathbb{P}^2})$ &
$\mathrm{G}_{\omega}(2,6)$ &
$\mathbb{OP}^2_0$ \\ 
$\AP$ & $v_2(\mathbb{P}^2)$ & $\mathbb{P}^2 \times \mathbb{P}^2$ &
$\mathrm{G}(2,6)$ &
$\mathbb{OP}^2$\\
$\GG$ & $\G_{\omega}(3,6)$ & $\G(3,6)$ & $\mathbb{S}_{12}$ & $\mathrm{E_7/P_7}$
\\
$E(\AA)^{ad}$ & $\mathrm{F_4^{ad}}$ & $\mathrm{E_6^{ad}}$ &
$\mathrm{E_7^{ad}} $&$
\mathrm{E_8^{ad}}$
\end{tabular} 
\end{center}
\bigskip

We are especially interested in the varieties in the third row. Let us give
another description of these varieties which is more concrete and which will
be useful for further computations. The space $W_{\AA}$ is naturally endowed
with a cubic form : the determinant. We will denote it by $C$. Thus, $C$ is a
linear form $S^3 W_{\AA} \rightarrow \mathbb{C}$ and can also be considered as a
linear map $S^2 W_{\AA} \rightarrow W_{\AA}^*$. We denote by $V_{\AA}$ the space
$\mathbb{C} \oplus W_{\AA} \oplus W_{\AA}^* \oplus \mathbb{C}^*$, which
coordinates are $(\a,A,B,\b)$. Denote by $\phi$ the rational map:

\begin{equation*}
\xymatrix{ 
\phi : \mathbb{P}(\mathbb{C} \oplus W_{\AA}) \ar@{^{}-->}[r] &
\mathbb{P}(V_{\AA}) \\
[\a:A] \ar[r] & [\frac{1}{6}{\a}^3 : {\a}^2A : {\a}C(A^{\otimes 2}) :
\frac{1}{3}C(A^{\otimes 3})]} \\ 
\end{equation*}
and denote by $\mathcal{Q}$ the quartic defined on $V_{\AA}$ by:
\begin{equation*}
\begin{split}
& \Q(\a,A,B,\b) = (3\a{\b} - \frac{1}{2} \langle A,B \rangle)^2 + \frac{1}{3}
\big( {\b}C(A^{\otimes 3}) + {\a}C^*(B^{\otimes 3}) \big) \\
& - \frac{1}{6}
\langle
C^*(B^{\otimes 2}), C(A^{\otimes 2}) \rangle,
\end{split}
\end{equation*}
where $\langle , \rangle $ is the natural pairing between $W_{\AA}$ and
$W_{\AA}^*$ and $C^*$ denotes the determinant on $W_{\AA}^*$. The equation of
the secant
variety to $\AP \subset \mathbb{P}(W_{\AA})$ is $\{C(A^{\ot 3}) = 0 \}$.
The
following result is proved in \cite{manilands}:

\begin{theo}
The variety $\G_{\omega}(\AA^3, \AA^6) \subset \mathbb{P}(V_{\AA})$ is the
image of the rational map $\phi$. The quartic $\Q$ is an
$\mathrm{Sp_6}(\AA)$-invariant form on $V_{\AA}$ and the
hypersurface $\Q = 0$ is the tangent variety
of $\G_{\omega}(\AA^3,\AA^6)$ in $\mathbb{P}(V_{\AA})$.
\end{theo}

\subsection{Desingularization of the tangent variety of
$\G_{\omega}(\AA^3,\AA^6)$}

The orbit stratification of the action of $\mathrm{Sp_6}(\AA)$ on
$\mathbb{P}(V_{\AA})$ is given as follows (in the upper parentheses, we let the
dimension of the corresponding orbit):

\begin{equation*}
\GG^{(3 \ma + 3)} \subset \sigma_+(\GG)^{(5 \ma +4)} \subset \tau(\GG)^{(6\ma +
6)} \subset \mathbb{P}(V_{\AA}),
\end{equation*}
where $\sigma_+(\GG)$ is the variety of stationary bisecants to $\GG$ and
$\tau(\GG)$ is the tangent variety to $\GG $. The singular locus of $\tau(\GG)$
is $\sigma_+(\GG)$ and the singular locus of $\sigma_+(\GG)$ is $\GG$. We refer
to \cite{manilands} for more details. One should however note that there is a
slight mistake in prop $5.10$ of \cite{manilands}. Indeed, $\GG$ is not the
triple locus of $\tau(\GG)$. One can check by a simple Taylor expansion of the
equation of $\tau(\GG)$ that the tangent cone to $\tau(\GG)$ at any point of
$\GG$ (for instance $[1:0:0:0]$) is a double hyperplane (this will be done
explicitly in the proof of Theorem \ref{desingtau}). Landsberg and Manivel
also provide explicit desingularizations of the varieties $\tau(\GG)$ and
$\sigma_+(\GG)$. The following propositions are proved in Section $7$ of
\cite{manilands}:

\begin{prop} \label{desingsigma}
There is a natural diagram:

\begin{equation*}
\xymatrix{ F = \tilde{Q}^{\dim \AA + 2} \ar[r]^{q} \ar@{^{(}->}[d] & \GG
\ar@{^{(}->}[d] \\
\mathbb{P}(\mathcal{S}) \ar[r]^{p} \ar[d]^{\theta} & \sigma_+(\GG) \\
\mathrm{Sp_6}^{ad}(\AA) & \\}
\end{equation*}
where the map $p : \mathbb{P}(\mathcal{S}) \rightarrow \sigma_+(\GG)$ is a
resolution of singularities.
\end{prop}
Here $\mathrm{Sp_6}^{ad}(\AA)$ is the closed orbit of the adjoint
representation  of $\mathrm{Sp_6}(\AA)$. The bundle $\mathcal{S}$ is a
homogeneous vector bundle on $\mathrm{Sp_6}^{ad}(\AA)$. The
map $\theta$ makes the exceptional divisor $F$ of $p$ a fibration into smooth
quadrics
of dimension $\dim \AA + 2$ over $\mathrm{Sp_6}^{ad}(\AA)$, while the map $q$
makes
it a fibration into $\AP$ over $\GG$. Note that the variety
$\mathbb{P}(\mathcal{S})$ is the blow-up of $\sigma_+(\GG)$ along $\GG$.

The orbit closure $\tau(\GG)$ can also be desingularized in a similar way:

\begin{prop} \label{desing}
Let $\tilde{T}\GG$ be the projective bundle of embedded tangent spaces to $\GG
\subset \mathbb{P}(V_{\AA})$. There is a natural diagram:

\begin{equation*}
\xymatrix{ E = \widetilde{\sigma(\AP)} \ar[r]^{\mu} \ar@{^{(}->}[d] &
\sigma_+(\GG) \ar@{^{(}->}[d] \\
\tilde{T}\GG \ar[r]^{\pi} \ar[d]^{\rho} & \tau(\GG) \\
\GG & \\}
\end{equation*}
where $\pi : \tilde{T}\GG \rightarrow \tau(\GG)$ is a resolution of
singularities.
\end{prop} 

The map $\rho : \tilde{T}\GG \rightarrow \GG$ is the projective bundle
whose fiber over $x \in \GG$ is the embedded projective tangent space to
$\GG$ at $x$. This map makes the exceptional divisor $E$ a fibration over
$\GG$ whose fibers are secant varieties of $\AP \subset \mathbb{P}(W_{\AA})$.
We denote it by $E = \widetilde{\sigma(\AP)}$.

The map $\mu : E \rightarrow \sigma_+(\GG)$ is not flat. Its fiber over $p \in
\GG$ is a cone over $\AP$, while its fiber over $p \in \sigma_+(\GG) \backslash
\GG$ is a smooth quadric of dimension $\dim \AA +1$. The map $\tilde{T}\GG
\rightarrow \tau(\GG)$ is the blow-up of $\tau(\GG)$ along $\sigma_+(\GG)$. 

The restriction of $\rho$ to the singular locus of the exceptional divisor $E$
makes it a fibration in $\AP$ over $\GG$. We denote it by $E_{sing} =
\widetilde{\AP}  \subset \widetilde{\sigma(\AP)}$. The divisor $E$ can be
desingularized by blowing up its singular locus and the desingularization is
also the projectivization of a homogeneous bundle. One also notices that
$\tilde{T}\GG$ is the blow-up of $\tau(\GG)$ along $\sigma_+(\GG)$. Since $\mu$
is smooth outside $\GG$, we have $E_{sing} \subset
\mu^{-1}(\GG)$. But a simple count of dimension shows that this inclusion is an
equality. We refer to \cite{manilands}, section $7$ for more details on this
desingularization.

\begin{rem} Though this resolution of singularities of $\tau(\GG)$ is quite
simple and very explicit, it will not be useful in order to find a categorical
crepant resolution of singularities of $\tau(\GG)$. Indeed, one of the key
points in order to construct such a categorical resolution would be to find a
semi-orthogonal decomposition:
\begin{equation*}
\begin{split}
 \DB(E) = & \langle \mu^* \DB(\sigma_+(\GG)) \ot \OO_{E}(r_{\AA}E), \ldots,
\mu^* \DB(\sigma_+(\GG)) \\
          & \ot \OO_{E}(E), \D \rangle,
\end{split}
\end{equation*}
where $E$ is the exceptional divisor of the resolution:
\begin{equation*}
\pi : \tilde{T}\GG \rightarrow \tau(\GG),
\end{equation*}
$r_{\AA}$ is the unique integer (well-defined since $E$ is integral) such that:
\begin{equation*}
K_{\tilde{T}\GG} = \pi^*K_{\tg} \ot \OO_{\tilde{T}\GG}(r_{\AA}E) 
\end{equation*}
and $\D$ is the left orthognal to the subcategory generated by the:
\begin{equation*}
\mu^* \DB(\sigma_+(\GG)) \ot \OO_{E}(kE),
\end{equation*}
for $1 \leq k \leq r_{\AA}$. Unfortunately the map:
\begin{equation*}
\mu : E  \rightarrow \sigma_+(\GG)
\end{equation*}
is not flat and $\sigma_+(\GG)$ is singular, thus $\mu^* \DB(\sigma_+(\GG))$
lies a priori in $\DM(E)$ and not in $\DB(E)$ (we prove in the Appendix A that
$\mu$ has infinite Tor-dimension, so that $\mu^* \DB(\sigma_{+}(\GG))$
really lies in $\DM(E)$ and not in $\DB(E)$). Though $\mu^*
\DM(\sigma_+(\GG))$ is an admissible subcategory of $\DM(E)$, it is very
unlikely (at least I am not able to prove it) that it is the negative completion
of an admissible subcategory of $\DB(E)$.

At this point, one could argue that the definition of a categorical crepant
resolution should be somehow modified and everything should be considered over
$\DM(\tau(\GG))$. Thus, a categorical resolution of $\tau(\GG)$ would be a
triangulated category $\T$, with a natural functor:
\begin{equation*}
{\pi_{\T}}_* : \T \rightarrow \DM(\tau(\GG)), 
\end{equation*}
such that $\T$ is an admissible subcategory of $\DM(Y)$, for some "geometric"
resolution of singularities $\pi :Y
\rightarrow
\tau(\GG)$. We should again have:
\begin{equation*}
\pi^* \DP(\tau(\GG)) \subset \T
\end{equation*}
and crepancy would be described as before:
\begin{equation*}
\pi_{\T}^* (\F) = \pi_{\T}^{!}(\F),
\end{equation*}
for all $\F \in \DP(\tau(\GG))$, where $\pi_{\T}^*$ and $\pi_{\T}^{!}$ are the
left and right adjoint of ${\pi_{\T}}_*$. However, this definition is not
meaningful
if one does not require that $\T$ comes from an admissible subcategory of
$\DB(Y)$. Otherwise, the theorem of Grauert-Riemenschneider would show that for
any resolution of singularities $\pi : Y \rightarrow \tau(\GG)$, the category
$\pi^* \DM(\tau(\GG))$ is always a categorical crepant resolution of
$\tau(\GG)$.
This is something we want to avoid, since we cannot consider $\pi^*
\DM(\tau(\GG))$ as a ``smooth`` triangulated category.

Hence, we see that we have to find another resolution of
singularities of $\tau(\GG)$, which would allow us to work over
$\DB(\tau(\GG))$.
\end{rem}

\begin{theo} \label{desingtau}
Let $\pi_1 : X_1 \rightarrow \tau(\GG)$ be the blow-up of $\tau(\GG)$ along
$\GG$ and let $\pi_2 : X_2 \rightarrow X_1$ be the blow-up of $X_1$ along the
strict transform of $\sigma_+(\GG)$ through $\pi_1$. The variety $X_2$ is a
resolution of singularities of $\tau(\GG)$.
\end{theo}

Note that the strict transform of $\sigma_+(\GG)$ through $\pi_1$ (which we
denote by $\pi_1^* \sigma_+(\GG)$) is the blow-up of $\sigma_+(\GG)$ along $\GG$
and it is smooth by proposition $\ref{desingsigma}$. As a consequence, the
sequence of blow-ups $\pi_1  : X_1 \rightarrow \tau(\GG)$ and $\pi_2 : X_2
\rightarrow X_1$ only consists of blow-ups along smooth centers (which we will
later prove to be normally flat). In such a case, the projection of any
exceptional divisor to the corresponding center of blow-up has finite
Tor-dimension, which will be very convenient for us.

Unfortunately, we are not able to describe this resolution as the total space of
a projective bundle over a flag variety. In fact, I believe that there is no
projective bundle over a flag variety whose total space coincide with $X_2$.
Thus, we have to check locally that
this sequence of blow-ups really produces a resolution of singularities. We
recall the equation of the tangent variety of $\GG \subset \mathbb{P}(V_{\AA})$:
\begin{equation*}
\begin{split}
& \Q(\a,A,B,\b) = (3\a \b - \frac{1}{2} \langle A,B \rangle)^2 + \frac{1}{3}
\big(
\b C(A^{\otimes 3}) + \a C^*(B^{\otimes 3}) \big) \\
& - \frac{1}{6} \langle
C^*(B^{\otimes 2}), C(A^{\otimes 2}) \rangle,\\
\end{split}
\end{equation*}
where $(\a,A,B,\b)$ is a system of coordinates for $V_{\AA} = \mathbb{C} \oplus
W_{\AA} \oplus W_{\AA}^* \oplus \mathbb{C}$. In the following we denote by $E_1$
the exceptional divisor of $\pi_1$, $E_2$ the exceptional divisor for $\pi_2$
and $E_1^{(2)}$ the strict transform of $E_1$ through $\pi_2$. Before diving
into the proof of theorem \ref{desingtau}, we introduce some more notations in
the diagrams below:

\bigskip
\bigskip

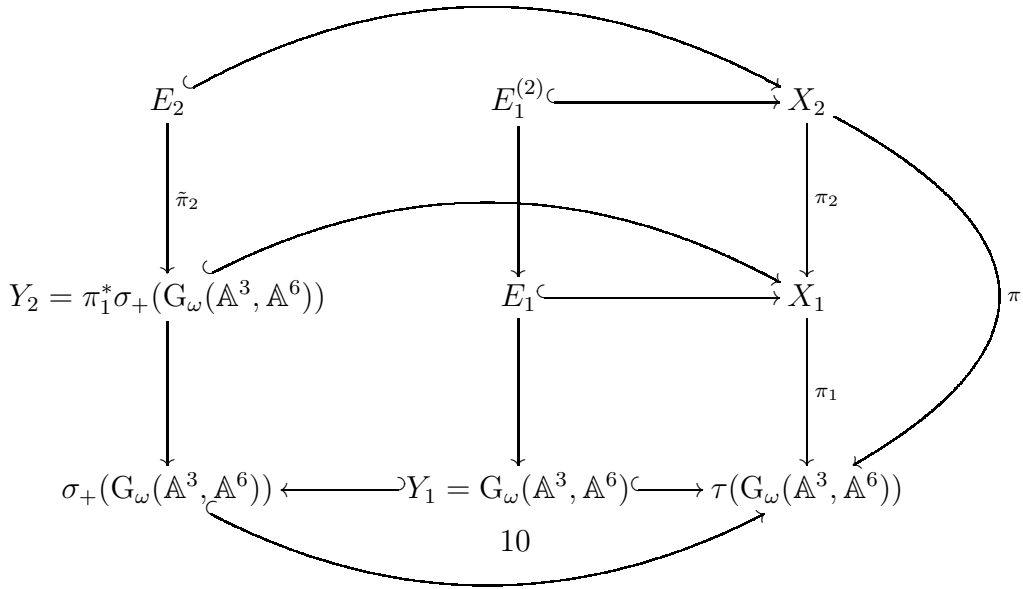
\begin{equation*}
\xymatrix{ 
E_2 \ar@{^{(}->}@/^3pc/[rr] \ar[dd]^{\tilde{\pi}_2} & E_1^{(2)} \ar@{^{(}->}[r]
\ar[dd] & X_2 \ar[dd]^{\pi_2}
\ar@/^6pc/[dddd]^{\pi} \\
& & \\
Y_2 = \pi_1^*\sigma_+(\GG) \ar[dd] \ar@{^{(}->}@/^3pc/[rr] & E_1 \ar@{^{(}->}[r]
\ar[dd] & X_1 \ar[dd]^{\pi_1} \\
& & \\
\sigma_+(\GG) \ar@{^{(}->}@/_3pc/[rr] & Y_1 = \GG \ar@{^{(}->}[r] 
\ar@{_{(}->}[l] & \tau(\GG)}
\end{equation*}
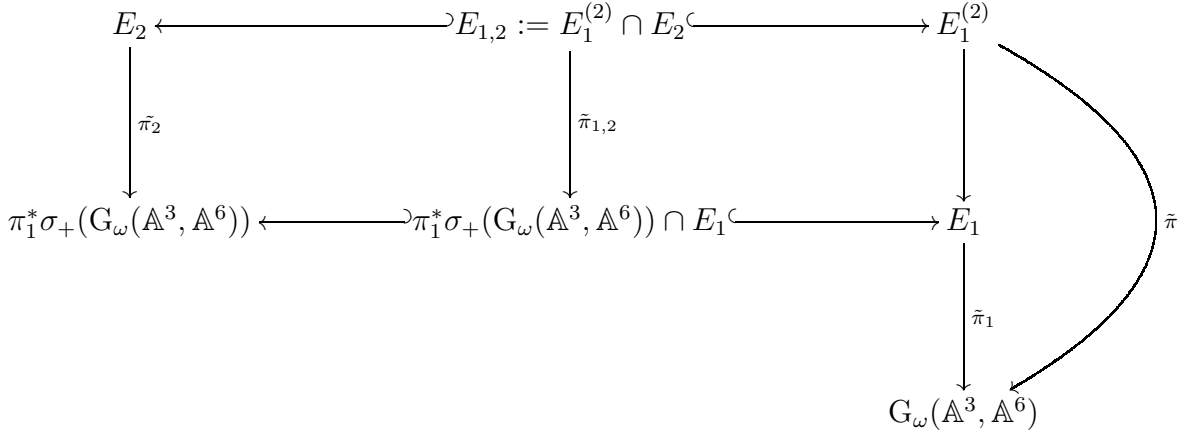
\bigskip
\bigskip
\bigskip

\begin{equation*}
\xymatrix{ 
E_2 \ar[dd]^{\tilde{\pi_2}} & & \ar@{_{(}->}[ll] E_{1,2} := E_1^{(2)}
\cap E_2 \ar@{^{(}->}[rr]
\ar[dd]^{\tilde{\pi}_{1,2}} & & E_1^{(2)} \ar[dd]
\ar@/^6pc/[dddd]^{\tilde{\pi}}
\\
& & \\
\pi_1^*\sigma_+(\GG) & & \ar@{_{(}->}[ll] \pi_1^*\sigma_+(\GG) \cap E_1
\ar@{^{(}->}[rr] & & E_1 \ar[dd]^{\tilde{\pi}_1} \\
& & & & \\
& & & & \GG}
\end{equation*}

\begin{proof} The proof of this result will be divided into several
steps.

\bigskip

\Large\textbf{Step 1 : Tangent cones to $\tg$ along its different orbits.}
\normalsize
\bigskip

We are going to compute the tangent cones to $\tau(\GG)$ at points of its
different strata. By $\mathrm{Sp_6}(\AA)$-equivariance, the hypersurface
$\tau(\GG)$ is normally flat along the orbit $\sigma_+(\GG)
\backslash \GG$. So, proposition \ref{desing} shows that the tangent cone
to $\tau(\GG)$ at any point $x \in \sigma_+(\GG) \backslash \GG$ is a cone over
a smooth quadric of dimension $\ma +1$ with vertex $T_{\sigma_+(\GG),x}$
(where $T_{\sigma_+(\GG),x}$ is the embedded tangent space to $\sigma_+(\GG)$
at $x$).

We also compute the tangent cone to $\tau(\GG)$ at $x \in \GG$. Since
$\tau(\GG)$
is invariant under the action of $\mathrm{Sp_6}(\AA)$ and $\GG$ is a closed
orbit in $\tau(\GG)$, we only need to compute the tangent cone at any given
point in $\GG$, say $x_0 = (1,0,0,0)$. The first partial derivatives of $\Q$ all
vanish at $x_0$ (because we know that $\sigma_+(\GG)$ is the singular locus of
$\tau(\GG)$. Furthermore, the polynomials $C(A^{\otimes 3})$ and $C^*(B^{\otimes
3})$ are
homogeneous cubic polynomials in the variables $A$ and $B$, thus we have (with a
slight abuse of notations):
\begin{equation*}
\frac{\partial^2 C(A^{\otimes 3})}{{\partial A}^2}(1,0,0,0) = \frac{\partial^2
C^*(B^{\otimes 3})}{{\partial B}^2}(1,0,0,0) = 0.
\end{equation*}
The polynomial $C(A^{\ot 2})$ and $C^*(B^{\ot 2})$ are homogeneous of
degree 2, thus we have:
\begin{equation*}
\frac{\partial^2 \langle C^*(B^{\otimes 2}), C(A^{\otimes 2}) \rangle}{{\partial
A}^2}(1,0,0,0) = \frac{\partial^2 \langle C^*(B^{\otimes 2}), C(A^{\otimes 2})
\rangle}{{\partial B}^2}(1,0,0,0) = 0.
\end{equation*}
The same type of arguments show that the only second partial derivative of $Q$
which does not vanish at $(1,0,0,0)$ is $\frac{\partial^2 \Q}{{\partial \b
}^2}(1,0,0,0) = 18$. Thus the tangent cone to $\tau(\GG)$ at $x_0$ is given by
the equation $18{\b}^2 = 0$, this is a double hyperplane. This means that $E_1$,
the exceptional divisor of $\pi_1 : X_1 \rightarrow \tau(\GG)$, is a fibration
into doubled $\mathbb{P}^{3 \ma +2}$ over $\GG$. Suppose that $|E_1|_{red}$ is a
Cartier divisor on $X_1$. Then $X_1$ is smooth along $|E_1|_{red}$ because
$|E_1|_{red}$ is smooth. But $\tau(\GG)$ is singular along $\sigma_+(\GG)$, so
that $X_1$ is singular along $\pi_1^{-1}(\sigma_+(\GG) \backslash \GG)$ (because
$\pi_1$ is an isomorphism outside $\GG$). By semi-continuity of the
multiplicity, $X_1$ is singular along the Zariski closure:
\begin{equation*}
\overline{ \pi_1^{-1}(\sigma_+(\GG) \backslash \GG)} = \pi_1^*(\sigma_+(\GG).
\end{equation*}
But $E_1 \cap \pi_1^*(\sigma_+(\GG))$ is not empty since it is the exceptional
divisor of the blow-up of $\sigma_+(\GG)$ along $\GG$. This is a contradiction
and shows that $|E_1|_{red}$ is not Cartier on $X_1$.

\bigskip 

The fact that $|E_1|_{red}$ is not Cartier on $X_1$ is a source of troubles.
Indeed, we cannot discuss the smothness of $X_1$ along $E_1 \backslash
E_1 \cap \pi_1^*\sigma_+(\GG)$. So we have to introduce an intermediate device
which enables us to prove the smoothness of $X_2$. 

Note that we proved that all
tangent cones to $\tg$ are at most quadratic, so that there is no point of
multiplicity strictly bigger than two in $\tg$.

\bigskip

\Large{\textbf{Step 2 : Resolution and polar divisors.}}

\large{\textbf{Step 2.1 : Strategy of the proof.}}
\normalsize

\bigskip

Let $p = (p_0,P_1,P_2,p_3) \in \mathbb{P}(V_{\AA})$ be a general point and let
$P(\Q,p)$ be the polar to $\tau(\GG)$ with respect to $p$, that is:
\begin{equation*}
P(\Q,p) = \tau(\GG) \cap \{ \H_p = 0 \},
\end{equation*}
where $\H_p = p_0\frac{\partial \Q}{\partial \a} + P_1\frac{\partial
\Q}{\partial A} + P_2\frac{\partial \Q}{\partial B} + p_3\frac{\partial
\Q}{\partial \b} = 0$. It is clear that $\sigma_+(\GG) = \tau(\GG)_{sing}
\subset P(\Q,p)$. Before going any further, we summarize the situation in the
following diagram:

\begin{equation*}
\xymatrix{ 
P^{(2)}(\Q,p) \ar[dd]^{q_2} & & \ar@{_{(}->}[ll] {E'_1}^{(2)} = E_1^{(2)} \cap
P^{(2)}(\Q,p) 
\ar@{^{(}->}[rr] \ar[dd]
 & & E_1^{(2)} \subset X_2 \ar[dd]^{\pi_2} 
\ar@/^6pc/[dddd]^{\pi}
\\
& & & & \\
P^{(1)}(\Q,p) \ar[dd]^{q_1} & & \ar@{_{(}->}[ll] E'_1 = E_1 \cap P^{(1)}(\Q,p)
\ar@{^{(}->}[rr] & & E_1 \subset X_1 \ar[dd]^{\pi_1}  \\
& & \\
P(\Q,p) \ar@{^{(}->}[rrrr] & & & & \tau(\GG)}
\end{equation*}

Our goal is to show that the strict transform of $P(\Q,p)$ through $\pi = \pi_1
\circ
\pi_2$ (which we denote by $ P^{(2)}(\Q,p))$ is smooth. Indeed, if we do so, we
get that $X_2$ is smooth along $P^{(2)}(\Q,p)$ (because $P^{(2)}(\Q,p)$ is a
Cartier divisor on $X_2$). Moreover, if we can prove that $X_2$ is smooth along
$E_2$ and along
\begin{equation*}
E_1^{(2)} \backslash \left( (P^{(2)}(\Q,p) \cup E_2) \cap
E_1^{(2)} \right) = E_1^{(2)} \backslash \left( {E'_1}^{(2)} \cup E_{1,2}
\right),
\end{equation*}
then we have won. Indeed, we already know that
$\tau(\GG)$ is smooth outside $P(\Q,p)$ so that $X_2$ is also smooth outside
\begin{equation*}
P^{(2)}(\Q,p) \cup E_2 \cup \left( E_1^{(2)} \backslash \left( {E'_1}^{(2)} \cup
E_{1,2} \right) \right)
\end{equation*}
because $\pi$ is an isomorphism outside this locus and we have:
\begin{equation*}
\pi \left( P^{(2)}(\Q,p) \cup E_2 \cup E_1^{(2)} \right) \subset P(\Q,p).
\end{equation*}
\bigskip

\large{\textbf{Step 2.1 : Smoothness along $P(\Q,p)^{(2)}$.}}

\normalsize \textbf{Step 2.2.a : Tangent cones to the polar divisors.}

\bigskip

First, we show that $\{ \H_p = 0 \}$ is a smooth cubic hypersurface.
Indeed, let $y \in \{\H_p =0 \}$ such that:
\begin{equation*}
\frac{\partial \H_p}{\partial \a}(y) = \frac{\partial \H_p}{\partial A}(y) =
\frac{\partial \H_p}{\partial B}(y) = \frac{\partial \H_p}{\partial \b}(y) = 0.
\end{equation*}
Since $p$ is a general point, the above equalities imply that all second
partial derivatives of $\Q$ vanish at $y$. But $\Q$ is a homogeneous polynomial,
so that $\Q$ and all its first partial derivatives also vanish at $y$. As a
consequence, $y$ is a point of multiplicity $3$ in $\tg$, which is impossible by
hypothesis.

\bigskip

Let us also prove that the tangent cone to $P(\Q,p)$ at any point $x \in
\sigma_+(\GG)
\backslash \GG$ is a cone over a smooth quadric of dimension $\ma$ with vertex
$T_{\sigma_+(\GG),x}$.

Let $x \in \sigma_+(\GG) \backslash \GG$, the tangent space to $\{ \H_p = 0 \}$
at $x$ is given by the equation:  
\begin{equation*}
\leftexp{t}{(\a,A,B,\b)} \left( \begin{array}{cccc}

\frac{\partial^2 \Q}{{\partial
\a}^2}(x) & \frac{\partial^2 \Q}{{\partial \a} \partial A}(x) & \frac{\partial^2
\Q}{{\partial \a} \partial B}(x) & \frac{\partial^2
\Q}{{\partial \a} \partial \b}(x) \\
\frac{\partial^2 \Q}{{\partial
A} \partial \a}(x) & \frac{\partial^2 \Q}{{\partial A}^2}(x) &
\frac{\partial^2
\Q}{{\partial A} \partial B}(x) & \frac{\partial^2
\Q}{{\partial A} \partial \b}(x) \\
\frac{\partial^2 \Q}{{\partial B} \partial \a}(x) & \frac{\partial^2
\Q}{{\partial B}
\partial A}(x) & \frac{\partial^2
\Q}{{\partial B}^2}(x) & \frac{\partial^2
\Q}{{\partial B} \partial \b}(x) \\
\frac{\partial^2 \Q}{{\partial \b} \partial \a}(x) & \frac{\partial^2
\Q}{{\partial \b}
\partial A}(x) & \frac{\partial^2
\Q}{{\partial \b} \partial B}(x) & \frac{\partial^2
\Q}{{\partial \b}^2}(x)
\end{array}
\right) (p_0,P_1,P_2,p_3) = 0
\end{equation*}
and the tangent cone to $\tg$ at $x$ is given by:

\begin{equation*}
\leftexp{t}{(\a,A,B,\b)} \left( \begin{array}{cccc}

\frac{\partial^2 \Q}{{\partial
\a}^2}(x) & \frac{\partial^2 \Q}{{\partial \a} \partial A}(x) & \frac{\partial^2
\Q}{{\partial \a} \partial B}(x) & \frac{\partial^2
\Q}{{\partial \a} \partial \b}(x) \\
\frac{\partial^2 \Q}{{\partial
A} \partial \a}(x) & \frac{\partial^2 \Q}{{\partial A}^2}(x) &
\frac{\partial^2
\Q}{{\partial A} \partial B}(x) & \frac{\partial^2
\Q}{{\partial A} \partial \b}(x) \\
\frac{\partial^2 \Q}{{\partial B} \partial \a}(x) & \frac{\partial^2
\Q}{{\partial B}
\partial A}(x) & \frac{\partial^2
\Q}{{\partial B}^2}(x) & \frac{\partial^2
\Q}{{\partial B} \partial \b}(x) \\
\frac{\partial^2 \Q}{{\partial \b} \partial \a}(x) & \frac{\partial^2
\Q}{{\partial \b}
\partial A}(x) & \frac{\partial^2
\Q}{{\partial \b} \partial B}(x) & \frac{\partial^2
\Q}{{\partial \b}^2}(x)
\end{array}
\right) (\a,A,B,\b) = 0.
\end{equation*}
But we already showed that the tangent cone to $\tg$ at any $x \in 
\sigma_+(\GG) \backslash \GG$ is a cone with vertex $T_{\sigma_+(\GG),x}$ over a
smooth quadric of dimension $\ma + 1$. From this we deduce two facts:
\begin{itemize}
\item the projective dual of this tangent cone is a smooth quadric in
$T_{\sigma_+(\GG),x}^{\perp}$,

\item the image of the Hessian matrix of
$\Q$ (seen as a map $\mathbb{P}(V_{\AA}) \rightarrow \mathbb{P}(V_{\AA})^*$) is
the whole $T_{\sigma_+(\GG),x}^{\perp}$.
\end{itemize}
As a consequence, since $p$ is general in $\mathbb{P}(V_{\AA})$, the point:
\begin{equation*}
 \left( \begin{array}{cccc}

\frac{\partial^2 \Q}{{\partial
\a}^2}(x) & \frac{\partial^2 \Q}{{\partial \a} \partial A}(x) & \frac{\partial^2
\Q}{{\partial \a} \partial B}(x) & \frac{\partial^2
\Q}{{\partial \a} \partial \b}(x) \\
\frac{\partial^2 \Q}{{\partial
A} \partial \a}(x) & \frac{\partial^2 \Q}{{\partial A}^2}(x) &
\frac{\partial^2
\Q}{{\partial A} \partial B}(x) & \frac{\partial^2
\Q}{{\partial A} \partial \b}(x) \\
\frac{\partial^2 \Q}{{\partial B} \partial \a}(x) & \frac{\partial^2
\Q}{{\partial B}
\partial A}(x) & \frac{\partial^2
\Q}{{\partial B}^2}(x) & \frac{\partial^2
\Q}{{\partial B} \partial \b}(x) \\
\frac{\partial^2 \Q}{{\partial \b} \partial \a}(x) & \frac{\partial^2
\Q}{{\partial \b}
\partial A}(x) & \frac{\partial^2
\Q}{{\partial \b} \partial B}(x) & \frac{\partial^2
\Q}{{\partial \b}^2}(x)
\end{array}
\right) (p_0,P_1,P_2,p_3)
\end{equation*}
does not lie in the projective dual to the tangent cone to $\tg$ at $x$. This
amounts to say that the intersection of $T_{ \{\H_p = 0 \},x}$ with
the tangent cone to $\tau(\GG)$ at $x$ is transverse. Hence, the tangent cone to
$P(\Q,p)$ at $x$ is a cone over a smooth quadric of dimension $\ma$ with vertex
$T_{\{ \H_p = 0 \},x}$.

\bigskip

Now, we are interested in the tangent
cone to $P(\Q,p)$ at $x \in \GG$. We will compute it at $x_0 = (1,0,0,0)$ for
simplicity. The Taylor expansion of $\Q$ at $x_0$ is:

\begin{equation*}
\Q(\a,A,B,\b) = 9{\b}^2 + \frac{1}{3}. C^*(B^{\ot 3}) -3\b \langle A, B \rangle
+\, \text{terms of
order $4$},
\end{equation*}
and the expansion of $\H_p$ at $x_0$ is:
\begin{equation*}
\H_p(\a,A,B,\b) =  18p_3 \b + \, \text{terms of order $2$}. 
\end{equation*}
The tangent cone to $P(\Q,p)$ is defined by the ideal generated by all the
leading forms of the equations in the ideal generated by $\Q$ and $\H_p$. Let $f
= 2p_3 \Q - (\b+\frac{\langle A,B \rangle}{3 p_3}) \H_p$. Then one checks that
the Taylor expansion of $f$ at                  
$x_0$ is:
\begin{equation*}
f(\a,A,B,\b) = \frac{2p_3}{3}. C^*(B^{\otimes^3}) + \b.(\,\text{terms of
order $2$} \,) + \, \text{terms of order $4$}.
\end{equation*}
As a consequence, the tangent cone to $P(\Q,p)$ at $x_0$ (which we denote
by $\C_{P(\Q,p),x_0}$) is given by the
equation $ \{ \b =0 \}$ and $\{ C^*(B^{\ot 3}) = 0 \}$. This is the cone
over the secant variety to $\AP \subset \mathbb{P}(V_{\AA}) =
|\tilde{\pi_1}^{-1}(x_0)|_{red}$ with vertex $T_{\GG,x_0}$. Notice that this
tangent cone does not depend on the general point $p$ choosen to define the
polar $P(\Q,p)$. Hence, by $\mathrm{Sp_6}(\AA)$-equivariance, this is true for
all $x \in \GG$. Thus, for all $x \in \GG$, the tangent cone $\C_{P(\Q,p),x}$ is
the cone over the secant variety $\AP \subset \mathbb{P}(V_{\AA}) =
|\tilde{\pi}_1^{-1}(x)|_{red}$ with vertex $T_{\GG,x}$.

\bigskip

\textbf{Step 2.2.b : Explicit resolution of the polar divisors.}

\bigskip

Let $q_1 : P(\Q,p)^{(1)} \rightarrow P(\Q,p)$ be the blow-up of $P(\Q,p)$ along
$\GG$ ($P(\Q,p)^{(1)}$ is also the strict transform of $P(\Q,p)$ along $\pi_1$)
and
denote by $E'_1$ the exceptional divisor of that blow-up. The above description
of the tangent cones of $P(\Q,p)$ at any $x \in \GG$ shows that the map
$q_1 : E'_1 \rightarrow \GG$ is a fibration into secant varieties to $\AP
\subset |\pi_1^{-1}(x)|_{red}$, for $x \in \GG$. Since $\AP \subset
\mathbb{P}(V_{\AA})$ is exactly the singular locus of its secant variety, the
singular locus
of $E'_1$ is a fibration into $\AP$ over $\GG$. Moreover, by Proposition
\ref{desingsigma}, the fiber over $x \in \GG$ of the exceptional divisor
of the blow-up of $\sigma_+(\GG)$ along $\GG$ is the secant variety to $\AP
\subset \mathbb{P}(V_{\AA}) = |\tilde{\pi_1}^{-1}(x)|_{red}$. Therefore, we
have:
\begin{equation*}
{E'_1}_{sing} = E'_1 \cap \pi_1^* \sigma_+(\GG) = E_1 \cap \pi_1^*
\sigma_+(\GG). 
\end{equation*}
Note that $P(\Q,p)^{(1)}$ is smooth along $E'_1 \backslash {E'_1}_{sing}$,
because $E'_1$ is a Cartier divisor on $P(\Q,p)^{(1)}$. We discussed the tangent
cones to $P(\Q,p)$ at points in $\sigma_+(\GG) \backslash \GG$ : these are cones
over smooth quadrics of dimension $\ma$ with vertex $T_{\sigma_+(\GG)}$. Thus
for
any $x$ in
\begin{equation*}
\pi_1^* \sigma_+(\GG) \backslash \left( E'_1 \cap \sigma_+(\GG)
\right),
\end{equation*}
the tangent cone to $P(\Q,p)^{(1)}$ at $x$ is again a cone over a
smooth quadric of
dimension $\ma$ with vertex $T_{\pi_1^*(\sigma_+(\GG)),x}$.

\bigskip

Let us compute the tangent cone to $P(\Q,p)^{(1)}$ at any point $x \in
{E'_1}_{sing}$. We know that $E'_1$ is a fibration into secant varieties of
$\AP$
over $\GG$. But the tangent cone to this secant variety at any point $x \in \AP$
is a cone over a smooth quadric of dimension $\ma$ with vertex $T_{\AP,x}$.
Hence, the tangent cone to $E'_1$ at $x \in {E'_1}_{sing}$ is a cone over a
smooth quadric of dimension $\ma$ with vertex $T_{{E'_1}_{sing},x}$. Since
$E'_1$ is a Cartier divisor in $P(\Q,p)^{(1)}$, we have:
\begin{equation*}
\mathrm{mult} \OO_{P(\Q,p)^{(1)},x} \leq \mathrm{mult} \OO_{E'_1,x} =2,
\end{equation*}
for any $x \in {E'_1}_{sing}$. Moreover, we know
that $\mathrm{mult} \OO_{P(\Q,p)^{(1)},y} = 2$ for all $y \in \pi_1^*
\sigma_+(\GG) \backslash {E'_1}_{sing}$. Thus, by semi-continuity of the
multiplicity, we have:
\begin{equation*}
\mathrm{mult} \OO_{P(\Q,p)^{(1)},x} = 2,
\end{equation*}
 for all $x \in {E'_1}_{sing}$. Since the tangent cone to $E'_1$ at $x \in
{E'_1}_{sing}$ is a cone over a smooth quadric of dimension $\ma$ with vertex
$T_{{E'_1}_{sing},x}$, we deduce that the tangent cone to $P(\Q,p)^{(1)}$ at $x$
is a cone over the same smooth quadric of dimension $\ma$, but with vertex
$T_{\pi_1^* \sigma_+(\GG),x}$ (recall that $\pi_1^* \sigma_+(\GG)$ is smooth by
proposition \ref{desingsigma}).

\bigskip

Let $q_2 : P(\Q,p)^{(2)} \rightarrow P(\Q,p)^{(1)}$ be the blow-up of
$P(\Q,p)^{(1)}$ along $\pi_1^*\sigma_+(\GG)$ ($P(\Q,p)^{(2)}$ is the strict
transform of $P(\Q,p)^{(1)}$ along $\pi_2$) and denote by $E'_2$ be the
exceptional divisor of that blow-up. The above description of the tangent cones
to $P(\Q,p)^{(1)}$ at any $x \in \pi_1^* \sigma_+(\GG)$ shows that the map $q_2
: E'_2 \rightarrow \pi_1^*\sigma_+(\GG)$ is a fibration into smooth quadrics of
dimension $\ma$. This implies that $E'_2$ is smooth, from which we deduce that
$P(\Q,p)^{(2)}$ is smooth along $E'_2$. Moreover, we proved that $P(\Q,p)^{(1)}$
is smooth along $E'_1 \backslash (E'_1 \cap \pi_1^* \sigma_+(\GG))$. As a
consequence, $P(\Q,p)^{(2)}$ is also smooth along ${E'_1}^{(2)}$, the total
transform of $E'_1$ through $q_2$. Since $P(\Q,p)$ is smooth outside
$\sigma_+(\GG)$, we get that $P(\Q,p)^{(2)}$ is also smooth outside $E'_2 \cup
{E'_1}^{(2)}$ and this completes the proof of the smoothness of $P(\Q,p)^{(2)}$.

\bigskip

Finally $\OO_{X_2}(P(\Q,p)^{(2)}) = \pi^* \OO_{X_2}(P(\Q,p)) \ot \OO_{X_2}(k_1
E_1^{(2)} + k_2 E_2)$, where $k_1$ and $k_2$ are some integers. We deduce that
$P(\Q,p)^{(2)}$ is a Cartier divisor in $X_2$. Hence the smoothness of
$P(\Q,p)^{(2)}$ implies the smoothness of $X_2$ along $P(\Q,p)^{(2)}$.
\bigskip 

\large{\textbf{Step 2.3 : Smoothness along $E_2$}}
\normalsize
\bigskip

The Cartier divisor $E_2 \subset X_2$ is a fibration into smooth quadrics of
dimension $\ma+1$ over $\pi_1^* \sigma_+(\GG)$, from which we deduce that it is
smooth. As a consequence, the variety $X_2$ is also smooth along $E_2$.

\bigskip
\large{ \textbf{Step 2.4 : ``Final step``: smoothness along $E_1^{(2)}
\backslash \\
\left( (P(\Q,p)^{(2)} \cup E_2) \cap E_1^{(2)} \right)$}}
\normalsize

\bigskip
In the following, we denote by $\mathrm{SL}_3(\AA)$, the groups :
$\mathrm{SL}_3$, $\mathrm{SL}_3\times \mathrm{SL}_3$, $\mathrm{SL}_6$ and
$\mathrm{E}_6$.

We finally show that $X_2$ is smooth. The only fact left to
demonstrate is
that $X_2$ is smooth along: 
\begin{equation*}
E_1^{(2)} \backslash
 \left( (P(\Q,p)^{(2)} \cup E_2) \cap E_1^{(2)} \right) = E_1^{(2)} \backslash
\left( {E'_1}^{(2)} \cup E_{1,2} \right).
\end{equation*}
To do so, we need to exploit the action of
$\mathrm{Sp_6}(\AA)$ on $\tau(\GG)$. The universal property of the blow-up
implies that the stabilizer of $x$ in $\mathrm{Sp_6}(\AA)$ acts on
$\pi_1^{-1}(x)$. The reductive part of this stabilizer is $\mathrm{SL_3(\AA)}$
(see \cite{manilands}).
Any non-trivial orbit closure of the action of this stabilizer on
$|\pi_1^{-1}(x)|_{red}$ is an orbit closure for the action of
$\mathrm{SL_3}(\AA)$ on $\mathbb{P}(W_{\AA})$. Hence,
the orbit diagram of the action on $|\pi_1^{-1}(x)|_{red}$ of the stabilizer of
$x$ in
$\mathrm{Sp_6}(\AA)$ is:
\begin{equation*}
\AP \subset \sigma(\AP) \subset \mathbb{P}(W_{\AA}) = |\pi_1^{-1}(x)|_{red}.
\end{equation*}
The group $\mathrm{Sp_6}(\AA)$ acts on $X_1$ and $E_1$ is stable under this
action. The above description of
the action on $|\pi_1^{-1}(x)|_{red}$ of the stabilizer of $x$ in
$\mathrm{Sp_6}(\AA)$
shows that the dense orbit in $|E_1|_{red}$ is the complement of
$P(\Q,p)^{(1)} \cap E_1
= E'_1$. The group
$\mathrm{Sp_6}(\AA)$ also acts on $X_2$ and $E_1^{(2)}$ is the stable
for this action. The dense orbit inside $|E_1^{(2)}|_{red}$ is the
complement in $|E_1^{(2)}|_{red}$ of ${E'_1}^{(2)} \cup E_{1,2}$. As a
consequence, the multiplicity of $X_2$ along $|E_1^{(2)}|_{red} \backslash
\left({E'_1}^{(2)} \cup E_{1,2} \right)$ is less than the multiplicity of $X_2$
along
$E_{1,2}$. But we know that $X_2$ is smooth along $E_2$, so that $X_2$ is smooth
along $E_1^{(2)} \backslash \left( {E'_1}^{(2)} \cup E_{1,2} \right)$ and we are
done!  
\end{proof}

In fact, we believe that a much more general statement than Theorem
\ref{desingtau} holds. To state our conjecture, we need some recollections on
prehomogeneous vector spaces (we refer to \cite{kimura} for a detailed treatment
of prehomogeneous spaces).

\begin{defi}
A \emph{strongly prehomogeneous vector space} is the data $(\mathrm{G},V)$ of an
algebraic group $\mathrm{G}$ acting linearly on a finite dimensional vector
space $V$ with a finite number of orbits. 

\noindent Let us denote by $V_{\mathrm{G}}^0,\ldots, V_{\mathrm{G}}^m$ the
orbits
of $\mathrm{G}$ on $V$. We say that the orbit diagram of $(\mathrm{G},V)$ is
\emph{linear} if $V_{\mathrm{G}}^0 = \{0 \}$ and up to
a reordering, we have:
\begin{equation*}
V_{\mathrm{G}}^i \subset \overline{V_{\mathrm{G}}^{i+1}},
\end{equation*} 
for all $i \geq 0$, where $\overline{V_{\mathrm{G}}^{i+1}}$ denotes the Zariski
closure of $V_{\mathrm{G}}^{i+1}$.
\end{defi}

\begin{exem}
\begin{itemize}
\item The square determinantal varieties of size $n$ are the orbits of the
action of $\mathrm{GL_n}\times \mathrm{GL_n}$ on $\mathbb{C}^n \ot
\mathbb{C}^n$. Their orbit diagram is linear.
\item The symmetric (resp. skew-symmetric) determinantal varieties of size $n$
are the orbits of the action of $\mathrm{GL_n}$ (resp. $\mathrm{GL_n}$) on $S^2
\mathbb{C}^n$ (resp. $\bigwedge^2 \mathbb{C}^n$). Their orbit diagram is also
linear.

\item The pair $(\mathbb{C}^* \times \mathrm{Sp_6}(\AA), V_{\AA})$ is a strongly
prehomogeneous vector space
whose orbit diagram is again linear.

\item The pair $(\mathrm{GL_8}, \bigwedge^3 \mathbb{C}^8)$ is a
strongly prehomogeneous
space whose orbit diagram is not linear (see \cite{holweck}).

\item The pair $(\mathrm{GL_9}, \bigwedge^3 \mathbb{C}^9)$ is not a
prehomogeneous space (see \cite{holweck}).
\end{itemize}
\end{exem}
We can now state our conjecture:

\begin{conj}
Let $(\mathrm{G},V)$ be a strongly prehomogenous vector space whose orbit
diagram $\{
V_{\mathrm{G}}^0,\ldots, V_{\mathrm{G}}^m \}$ is linear and let $X =
\mathbb{P}(\overline{V_{\mathrm{G}}^i})$ be the projectivization of the closure
of any orbit. Consider the sequence:
\begin{equation*}
X_i \stackrel{\pi_i}\rightarrow X_{i-1} \rightarrow \cdots \rightarrow X_2
\stackrel{\pi_2}\rightarrow
X_1 = X,
\end{equation*}
where $\pi_k : X_k \rightarrow X_{k-1}$ is the blow-up of the strict transform
of $\mathbb{P}(\overline{V_{\mathrm{G}}^{k-1}})$ through $\pi_1 \circ \ldots
\circ
\pi_{k-1}$. Then $X_i$ is smooth.
\end{conj}

This conjecture is well-known for all square, symmetric and skew-symmetric
determinantal varieties (see Example $2.1.3$ of \cite{abuafcategorical}).
Theorem \ref{desingtau} and Proposition \ref{desingsigma} show
that the conjecture holds for the pair $(\mathbb{C}^* \times \mathrm{Sp_6}(\AA),
V_{\AA})$.

\subsection{Some vanishing lemmas}
In this section, we state the vanishing lemmas we will need for the proof of our
main theorem. Recall that $E_1$, the exceptional divisor of the map $\pi_1 :
X_1 \rightarrow \tau(\GG)$, is a fibration into doubled $\mathbb{P}^{3\ma +2}$
over $\GG$. As for $E_2$, the exceptional divisor of the map $\pi_2 : X_2
\rightarrow X_1$, it is a fibration in smooth quadrics of dimension $\ma + 1$
over the strict transform of $\sigma_+(\GG)$ through $\pi_1$. 

We also recall some notations we used in the proof of Theorem
\ref{desingtau}.
$\e12$ denotes the total transform of $E_1$ through $\pi_2$. Since the
intersection of the proper transform of $\sigma_+(\GG)$ through $\pi_1$ (which
we denote by $\pi_1^*\sigma_+(\GG)$) with $E_1$ is proper, the divisor $\e12$ is
also the blow-up of $E_1$ along $\pi_1^*\sigma_+(\GG) \cap E_1$. The divisor
$E_{1,2}$ the intersection $\e12 \cap E_2$, which is also the exceptional
divisor of the blow-up of $E_1$ along $\pi_1^*\sigma_+(\GG) \cap E_1$. The
morphism $\pi$ is the composition $\pi_1 \circ \pi_2$ and $\pi_{1,2}$ is the
restriction
of $\pi_2$ to $\e12$. Finally, we denote by $\tilde{\pi}_1$ (resp.
$\tilde{\pi}_2, \tilde{\pi}$ and $\tilde{\pi}_{1,2}$) the restriction of
$\pi_1$
(resp. $\pi_2, \pi$ and $\pi_{1,2}$) to $E_1$ (resp. $E_2, \e12$ and
$E_{1,2}$). We summarize these notations in the following diagrams (which
already appeared in the proof of \ref{desingtau}:

\begin{equation*}
\xymatrix{ 
E_2 \ar@{^{(}->}[rr]^{i_2} \ar[dd]^{\tilde{\pi}_2} & & X_2 \ar[dd]^{\pi_2}
\ar@/^6pc/[dddd]^{\pi} \\
& & \\
\pi_1^*\sigma_+(\GG) \ar@{^{(}->}[rr]^{j_2} & & X_1 \ar[dd]^{\pi_1} \\
& & \\
& & \tau(\GG)}
\end{equation*}

\begin{equation*}
\xymatrix{ 
E_2 \ar[dd]^{\tilde{\pi}_2} & & \ar@{_{(}->}[ll]_{i_{2,1}} E_{1,2}
\ar@{^{(}->}[rr]^{i_{1,2}}
\ar[dd]^{\tilde{\pi}_{1,2}} & & E_1^{(2)} \ar[dd]^{\pi_{1,2}}
\ar@/^6pc/[dddd]^{\tilde{\pi}}
\\
& & \\
\pi_1^*\sigma_+(\GG) & & \ar@{_{(}->}[ll]_{j_2} \pi_1^*\sigma_+(\GG) \cap E_1
\ar@{^{(}->}[rr]^{\kappa_1} & & E_1 \ar[dd]^{\tilde{\pi}_1} \\
& & & & \\
& & & & \GG}
\end{equation*}

We start with transformation formulas for the canonical bundle through the maps
$\pi_1$ and $\pi_2$.

\begin{lem} \label{formula-canonicalII}
We have the formulas:
\begin{equation*}
\w_{X_1} = \pi_1^*\w_{\tau(\GG)} \ot \OO_{X_1}((3\ma +1)E_1),
\end{equation*}
and
\begin{equation*}
\w_{X_2} = \pi_2^*\w_{X_1} \ot \OO_{X_2}(\ma E_2).
\end{equation*}
\end{lem}  

The existence of an integer $p$ such that $\w_{X_2} = \pi_2^*\w_{X_1} \ot
\OO_{X_2}(pE_2)$ is trivial as $E_2$, the scheme-theoretic exceptional locus
of $\pi_2$, is an integral divisor on $X_2$. The existence of such a formula for
$\w_{X_1}$ is less obvious. Indeed, since $E_1$ is not reduced, one could
imagine
an equality $\w_{X_1} = \pi_1^*\w_{\tau(\GG)} \ot \OO_{X_1}(qE_1')$, where
$E_1'$
is a Cartier divisor on $X_1$ with $|E_1|_{red} = |E_1'|_{red}$, but such that
$qE_1'$ is not a multiple of $E_1$.

\begin{proof}
We start with the formula for $\w_{X_1}$. We divide the proof of this formula
into two steps:
\begin{itemize}
\item we prove that the blow-up $ \pi_1 : X_1 \rightarrow \tau(\GG)$ is the
contraction of a negative extremal ray (see \cite{mori-kollar}, section $3$),
\item we prove that the bundle $\w_{X_1} \ot \pi_1^*\w_{\tau(\GG)}^{-1} \ot
\OO_{X_1}((-3\ma -1)E_1)$ is trivial.
\end{itemize}

\textbf{Step 1 : The blow-up $X_1 \rightarrow \tau(\GG)$ is a Mori
contraction.}

\bigskip

\noindent Let $\widetilde{\mathbb{P}(V_{\AA})}$
be the blow-up of $\mathbb{P}(V_{\AA})$ along $\GG$ and denote by $H_1$ the
exceptional divisor of that blow-up. We have $E_1 = H_1|_{X_1}$. The map $ q_1 :
H_1 \rightarrow \GG$ is a projective bundle of relative dimension $3\ma +3$ over
$\GG$ and the restriction of $H_1$ to any fiber $q_1^{-1}(x) =
\mathbb{P}^{3\ma+3}$ is $\OO_{\mathbb{P}^{3\ma+3}}(1)$. As a consequence, we
have the equality:
\begin{equation*}
{E_1}|_{\pi_1^{-1}(x)} = \OO_{\mathbb{P}^{3\ma+3}}(1)|_{\pi_1^{-1}(x)},
\end{equation*}
for all $x \in \GG$. We will denote this last bundle by
$\OO_{\pi_1^{-1}(x)}(1)$. Recall that the
proof of Theorem \ref{desingtau} shows that $\pi_1^{-1}(x)$ is a doubled
$\mathbb{P}^{3\ma +2}$ in $q_1^{-1}(x)$. Thus we have:
\begin{equation*}
\w_{\pi_1^{-1}(x)} = \OO_{\pi_1^{-1}}(-3\ma -2).
\end{equation*}
By the adjunction formula we have $\w_{E_1} = \w_{X_1} \ot \OO_{E_1}(E_1)$. The
morphism $\tilde{\pi}_1 : E_1 \rightarrow \GG$ is flat and $\GG$ is smooth so
that the normal bundle $N_{\tilde{\pi}_1^{-1}(x)/E_1}$ is trivial. By the
adjunction formula, we get: 
\begin{equation*}
\w_{E_1}|_{\tilde{\pi}_1^{-1}(x)} = \w_{\tilde{\pi}_1^{-1}(x)} =
\OO_{\tilde{\pi}_1^{-1}(x)}(-3\ma -2)
\end{equation*}
Let $NE_{\pi_1}(X)$ be the cone of effective $1$-cycles in $X$ contracted by
$\pi_1$ and let $R \in NE_{\pi_1}(X)$ be the numerical class of a line in
$\pi_1^{-1}(x)$. The above formula shows that:
\begin{equation*}
\w_{X_1}.R <0,
\end{equation*}
so that $R$ is a negative ray for $X_1$ with respect to $\pi_1$. Let us prove
that $NE_{\pi_1}(X) = \langle R \rangle.$ We have an exact sequence:

\begin{equation*}
0 \rightarrow \OO_{\mathbb{P}^{3\ma + 2}}(-1) \rightarrow \OO_{2\mathbb{P}^{3\ma
+ 2}} \rightarrow  \OO_{\mathbb{P}^{3\ma + 2}} \rightarrow 0,
\end{equation*}
where $\OO_{2\mathbb{P}^{3\ma + 2}}$ is the structure sheaf of a doubled
$\mathbb{P}^{3\ma+2}$ in a $\mathbb{P}^{3\ma+3}$. Note that
$\OO_{\mathbb{P}^{3\ma + 2}}(-1)$ consists of nilpotent elements of
$\OO_{2\mathbb{P}^{3\ma
+ 2}}$, so we can lift the above exact sequence to an exact sequence of groups
sheaves (see \cite{SGA2}, Exposé XI, section $1$):
\begin{equation*}
0 \rightarrow \OO_{\mathbb{P}^{3\ma + 2}}(-1) \rightarrow \OO_{2\mathbb{P}^{3\ma
+ 2}}^{\times} \rightarrow  \OO_{\mathbb{P}^{3\ma + 2}}^{\times} \rightarrow 1,
\end{equation*}
where $\OO_{X}^{\times}$ is the sheaf of units of the scheme $X$. Taking the
long exact sequence of cohomology, we find that:
\begin{equation*}
H^1(2\mathbb{P}^{3\ma +2}, \OO_{2\mathbb{P}^{3\ma + 2}}^{\times}) =
H^1(\mathbb{P}^{3\ma +2}, \OO_{\mathbb{P}^{3\ma
+ 2}}^{\times}),
\end{equation*}
that is:
\begin{equation*}
\mathrm{Pic}(2\mathbb{P}^{3\ma +2}) = \mathrm{Pic}(\mathbb{P}^{3\ma +2}) =
\mathbb{Z}. 
\end{equation*}
Thus, we see that the cone of effective $1$-cycles (modulo numerical
equivalence) on $\pi_1^{-1}(x)$ is of dimension $1$. Since the morphism
$\tilde{\pi}_1 : E_1 \rightarrow \sigma_+(\GG)$ is flat, the cone
$NE_{\pi_1}(X)$ is also of dimension $1$. Hence, we have:
\begin{equation*}
NE_{\pi_1}(X) = \langle R \rangle,
\end{equation*}
so that $R$ is a negative extremal ray for $NE_{\pi_1}(X)$. As $X_1$ is
Gorenstein with rational
singularities (hence canonical singularities, by \cite{kollar} proposition
$11.13$), we can apply the relative Cone theorem to $X_1$ and $R$ (see
\cite{mori-kollar}, Theorem $3.25$) and we find a commutative diagram:

\begin{equation*}
\xymatrix{ 
X_1 \ar[ddr]^{\pi_1} \ar[rr]^{p} & & Y \ar[ddl]^{q} \\
& & \\
 & \tau(\GG) & \\}
\end{equation*}
where $p$ is the contraction of the extremal ray $R$. We know that
$NE_{\pi_1}(X) = \langle R \rangle$. Therefore, for all $x \in \GG$, all the
curves lying in $\pi_1^{-1}(x)$ are contracted by $p$.

We want to demonstrate that $q$ is an isomorphism. Let $x \in \GG$ and assume
that $\dim p(\pi_1^{-1}(x)) > 0$. We can find two curves $C \in
p(\pi_1^{-1}(x))$ and $C' \subset \pi_1^{-1}(x)$ such that $p(C') = C$. But this
is a contradiction since all curves lying in $\pi_1^{-1}(x)$ are contracted by
$p$. As a consequence, for all $x \in \GG$, the scheme $p(\pi_1^{-1}(x))$ is
a point. We deduce that $q : Y \rightarrow \tau(\GG)$ is a birational finite
morphism such that $\RR q_* \OO_{Y} = \OO_{\tau(\GG)}$. The variety $\tau(\GG)$
is normal, so by the Main Theorem of Zariski, the morphism $q$ is an
isomorphism. We deduce that $\pi_1 : X_1 \rightarrow \tau(\GG)$ is the
contraction of the extremal ray
generated by $R$. 

\bigskip
\textbf{Step 2 : The bundle $\w_{X_1} \ot \pi_1^*\w_{\tau(\GG)}^{-1} \ot
\OO_{X_1}((-3\ma -1)E_1)$ is trivial.}

\bigskip

\noindent Let $L = \w_{X_1} \ot \OO_{X_1}((-3\ma-1)E_1)$. The formulae for the
restrictions of $\w_{X_1}$ to $E_1$ and $\w_{E_1}$ to $\pi_1^{-1}(x)$ show that
$L_{\pi_1^{-1}(x)} = \OO_{\pi_1^{-1}(x)}$, for any $x \in \GG$. Thus, we
can apply again the relative Cone Theorem and we get $L = \pi_1^* L'$ for some
line bundle $L'$ on $\tau(\GG)$.

Let us prove that $L' = \w_{\tau\GG}$. Since $\RR {\pi_1}_* \OO_{X_1} =
\OO_{\tau(\GG)}$ and $\dim X_1 = \dim \tau(\GG)$, Grothendieck duality shows
that the complex $\RR {\pi_1}_* \w_{X_1}[\dim X_1]$ is a dualizing complex for
$\tau(\GG)$. We apply the
Grauert-Riemenschneider theorem to $\pi_1$ and we get $\Ri {\pi_1}_* \w_{X_1} =
0$ for $i>0$. As a consequence, we have ${\pi_1}_* \w_{X_1} = \w_{\tau(\GG)}$.
Moreover, the divisor $E_1$ is effective and contracted by $\pi_1$. The variety
$\tau(\GG)$ being normal and $\pi_1$ being birational, the sheaf $\pi_*
\OO_{X_1}((3 \ma +2)E_1)$ is trivial. Finally, we apply ${\pi_1}_*$ on both
sides of
the equality : 
\begin{equation*}
\w_{X_1} = \pi_1^*(L') \ot \OO_{X_1}(3\ma +1),
\end{equation*}
 and the projection formula gives:
\begin{equation*}
{\pi_1}_* \w_{X_1} = L'. 
\end{equation*}
As we showed that ${\pi_1}_* \w_{X_1} = \w_{\tau(\GG)}$, this concludes the
proof that:
\begin{equation*}
\w_{X_1} = \pi_1^*\w_{\tau(\GG)} \ot \OO_{X_1}((3\ma+1)E_1).
\end{equation*}

\bigskip

The formula for $\w_{X_2}$ is proved in a similar fashion, but is much easier.
Indeed, as $E_2$ is the scheme-theoretic exceptional locus of $\pi_2$ and is
an integral divisor on $X_2$, there exists an integer $p$ such that $\w_{X_2} =
\pi_2^*\w_{X_1} \ot \OO_{X_2}(pE_2)$. We determine the integer $p$ by
restricting
this equality to the fibers of $\pi_2|_{E_2} : E_2 \rightarrow \pi^*_1
\sigma_+(\GG)$. The adjunction formula shows again that
\begin{equation*}
\w_{X_2} \ot \OO_{E_2}(E_2) \ot \OO_{\tilde{\pi_2}^{-1}(x)} =
\w_{\tilde{\pi_2}^{-1}(x)}
\end{equation*}
and we conclude using the fact that $\tilde{\pi_2}^{-1}(x)$ is a smooth quadric
of dimension $\ma +1$.
\end{proof}

\begin{prop} \label{vanishingII}
We have the vanishings:

\itemize
\item ${\Ri \tilde{\pi}_{2}}_* \OO_{E_2}(kE_2) = 0$, for all $i \geq 0$ and for
all $1 \leq k \leq \ma$,
\item $\Ri \tilde{\pi}_* \OO_{\e12}(k\e12) = 0$, for all $i \geq 0$ and all $1
\leq k \leq 3\ma+1$,   
\item ${\Ri {\tilde{\pi}_{1,2}}}_* \OO_{E_{1,2}}(kE_{1,2})$, for all $i \geq 0$
and for all $1 \leq k \leq \ma$.
\end{prop}

\begin{proof}
The three points are more or less direct consequences of the Kawamata-Viehweg
relative vanishing theorem and of the Grauert-Riemenschneider vanishing
theorem. 

For the first point, we have $\w_{X_2} = \pi_2^*\w_{X_1} \ot \OO_{X_2}(\ma E_2)$
and $-E_2$ is relatively ample with respect to $\pi_2$. Since $X_2$ is
Gorenstein with rational singularities (in fact it is smooth), we apply the
Kawamata-Viehweg relative vanishing theorem and we get:
\begin{equation*}
\Ri {\pi_2}_* \OO_{X_2}(kE_2) = 0,
\end{equation*}
for all $i>0$ and for all $k<\ma$. The vanishing:
\begin{equation*}
\Ri{\pi_2}_* \OO_{X_2}(\ma E_2) = 0,
\end{equation*}
for all $i>0$, is a consequence of the theorem of Grauert-Riemenschneider.

Now, for all $k \in \mathbb{Z}$, we have an exact sequence:
\begin{equation*}
0 \rightarrow \OO_{X_2}((k-1)E_2) \rightarrow \OO_{X_2}(kE_2) \rightarrow
\OO_{E_2}(kE_2) \rightarrow 0.
\end{equation*}
We take the long exact sequence associated to the functor $\RR {\pi_2}_*$ and
taking into account the above vanishing, we find:
\begin{equation*}
{\Ri{\tilde{\pi}_2}}_* \OO_{E_2}(kE_2) = 0,
\end{equation*}
for all $i>0$ and for all $k \leq \ma$.

\bigskip

\noindent Finally,we want to prove some vanishing for $\tilde{\pi}_{2*}
\OO_{E_2}(kE_2)$. Notice first that $E_2$ si an effective divisor contracted by
the birationnal morphism $\pi_2$. The variety $X_1$ being normal, we have:
\begin{equation*}
{\pi_2}_* \OO_{X_2}(kE_2) = \OO_{X_1},
\end{equation*}
for all $k \geq 0$. Thus, the long exact sequence associated to the
above short exact sequence and the vanishing results already proved imply:
\begin{equation*}
\tilde{\pi}_{2*} \OO_{E_2}(kE_2) = 0,
\end{equation*}
for all $k \geq 1$. This concludes the first point.

\bigskip

The second point is proved in the same manner with the following observation. We
have ${\RR {\pi_2}}_* \OO_{X_2}(\e12) = \OO_{X_1}(E_1)$ by the projection
formula. Thus, to prove the
vanishing result for $\Ri {\tilde{\pi}}_* \OO_{\e12}(k\e12)$, it is sufficient
to prove it for ${\Ri {\tilde{\pi}_1}}_* \OO_{E_1}(kE_1)$. This is done exactly
in
the same way as for the first point of the proposition.

\bigskip

The third point needs a slightly more involved argument. The intersection $\e12
\cap E_2 = E_{1,2}$ is proper, so we have a resolution:
\begin{equation*}
 0 \rightarrow \OO_{X_2}(-\e12-E_2) \rightarrow
\OO_{X_2}(-\e12)\oplus\OO_{X_2}(-E_2) \rightarrow \OO_{X_2} \rightarrow
\OO_{E_{1,2}} \rightarrow 0.
\end{equation*}
We tensor this resolution by $\OO_{X_2}(kE_2)$, for any integer $k$, and we get:
\begin{equation*}
\begin{split}
& 0 \rightarrow \OO_{X_2}(-\e12+(k-1)E_2) \rightarrow
\OO_{X_2}(-\e12)\oplus\OO_{X_2}((k-1)E_2) \rightarrow \OO_{X_2}(kE_2) \\
& \rightarrow \OO_{E_{1,2}}(kE_{1,2}) \rightarrow 0 \\
\end{split}
\end{equation*}
Recall that the Kawamata-Viehweg relative vanishing theorem and the
Grauert-Riemenschneider vanishing theorem imply that
\begin{equation*}
\Ri {\pi_{2}}_* \OO_{X_2}(kE_2) = 0,
\end{equation*}
for all $i>0$ and all $k \leq \ma$. Finally, we chop the above resolution into
two shorts exact sequences. We take the long exact sequences associated to the
functor $\RR {\pi_2}_*$ for these two short exact sequences and we find:
\begin{equation*}
{\Ri {\tilde{\pi}_{1,2}}}_* \OO_{E_{1,2}}(kE_{1,2}) = 0,
\end{equation*}
for all $i>0$ and all $k \leq \ma$.

\bigskip

The vanishing:
\begin{equation*}
\tilde{\pi}_{1,2*} \OO_{E_{1,2}}(kE_{1,2}) = 0,
\end{equation*}
for all $k \geq 1$ is proved as for the first point of the proposition. Indeed,
we have:
\begin{equation*}
{\pi_2}_* \OO_{X_2}(kE_2) = \OO_{X_1},
\end{equation*}
for all $k \geq 0$. We again chop the above long exact sequence into two short
exact sequences and we go on as in the proof of the first point of the
proposition.

\end{proof}

\section{Proof of the main theorem}
In this section we are going to prove our main result : 

\begin{theo} \label{mainII}
The variety $\tg$ admits a categorical crepant resolution of singularities. 
\end{theo}

From now on,for any proper morphism $f : X \rightarrow Y$ of schemes of finite
type, we denote by $f_*$ the total derived functor $\RR f_* : \DB (X)
\rightarrow \DB(Y) $, by $f^*$ the total derived functor $\LL f^*
: \mathrm{D^{-}}(Y) \rightarrow \mathrm{D^{-}}(X)$ and by $f^!$ the right
adjoint to $\RR f_* : \DB(X) \rightarrow \DB(Y)$. In case we need to use
specific homology sheaves of these functors, we will denote them by $\Ri f_*,
\Li f^*$ and $\Li f^!$. If $\F, \mathcal{G}$ are two objects of $\DM(X)$, we
denote by $\F \ot \mathcal{G}$ the derived tensor product $\F \ot^L
\mathcal{G}$.

\subsection{Standard reductions}
 Denote by $i_1 : \e12 \hookrightarrow X_2$ and $i_2 : E_2 \hookrightarrow X_2$
the embeddings of the exceptional divisors. We define the following
subcategories of $\DB(X_2)$:
\begin{equation*}
\B_k = {i_2}_* \left( \tilde{\pi}_2^*\DB(\pi_1^*\sigma_+(\GG)) \ot
\OO_{E_2}(kE_2) \right),
\end{equation*}
for all $1 \leq k \leq \ma$ and:
\begin{equation*}
\A_l = {i_1}_* \left( \tilde{\pi}^*\DB(\GG) \ot \OO_{\e12}(l \e12 + \ma E_2)
\right),
\end{equation*}
for all $1 \leq l \leq 3\ma+1$. Our key proposition is the following:

\begin{prop} \label{key}
We have a semi-orthogonal decomposition:

\begin{equation*}
\DB(X_2) = \langle \A_{3\ma+1},\ldots, \A_1, \B_{\ma},\ldots, \B_1, \D_{X_2}
\rangle,
\end{equation*}
where $\D_{X_2}$ is the left orthogonal to the full admissible subcategory
generated by the $\A_l$ and $\B_k$. Moreover we have the property:
\begin{equation*}
\pi^*\DP(\tau(\GG)) \subset \D_{X_2}.
\end{equation*}
\end{prop}

Before diving into the proof of Proposition \ref{key}, we explain how it implies
our main result. We will prove that $\D_{X_2}$ is a categorical crepant
resolution of singularities of $\tau(\GG)$. 

\begin{proof}[Proof of theorem \ref{mainII}]

First note that $\D_{X_2}$ is an admissible subcategory of
$\DB(X_2)$ and that $\pi^* \DP(\tau(\GG)) \subset \D_{X_2}$. Thus, we only have
to prove that for all $\F \in \DP(\tau(\GG))$, we have:

\begin{equation*}
\pi_{\D}^*(\F) \simeq \pi_{\D}^{!}(\F),
\end{equation*}
where $\pi_{\D}^*$ and $\pi_{\D}^!$ are the left and right adjoints to the
functor 
\begin{equation*}
\pi_{\D} : \D_{X_2} \rightarrow \DB(\tau(\GG)).
\end{equation*}
Let $\delta : \D_{X_2} \hookrightarrow \DB (X_2)$ be the fully faithful
admissible
embedding. We must prove that $\d^{*} \pi^{*}(\F) = \d^{!} \pi^{!}(\F)$, for all
$\F \in \DP
(\tg)$. Recall that
\begin{equation*}
\pi^!(\F) = \pi^*(\F) \ot \pi^*(\w_{\tau(\GG)}^{-1})\ot \w_{X_2} = \pi^*(\F) \ot
\OO_{X_2}(\ma E_2 + (3\ma +1)\e12).
\end{equation*}
Now, since the functor $\d$ is fully faithful, the equality $\d^{*} \pi^{*}(\F)
= \d^{!} \pi^{!}(\F)$ is equivalent to $\d(\d^*
\pi^*(\F)) =
\d (\d^! \pi^!(\F))$. As $\pi^* \DP (\tg) \subset
\D_{X_2}$,
we have $\d(\d^* \pi^*(\F)) = \pi^*(\F)$. We are going
to
show that $\d (\d^! \pi^!(\F)) = \pi^*(\F)$. 

\bigskip

For $1 \leq k \leq \ma$ and for $1 \leq l \leq 3\ma+1$, we have exact
sequences:

\begin{equation*}
\begin{split}
& 0 \rightarrow \OO_{X_2}((k-1)E_2) \rightarrow \OO_{X_2}(kE_2) \rightarrow
{i_2}_* \OO_{E_2}(kE_2) \rightarrow 0,\\
& 0 \rightarrow \OO_{X_2}((l-1)\e12 + \ma E_2) \rightarrow \OO_{X_2}(l\e12 + \ma
E_2) \rightarrow {i_1}_*\OO_{\e12}(l\e12 + \ma E_2) \rightarrow 0.
\end{split}
\end{equation*}
Tensoring the above exact sequences with $\pi^* \F$, we get
exact triangles:

\begin{equation*}
\begin{split}
& \OO_{X_2}((k-1)E_2) \ot \pi^*\F \rightarrow \OO_{X_2}(kE_2) \ot \pi^*\F
\rightarrow {i_2}_* \OO_{E_2}(kE_2) \ot \pi^*\F,\\
& \OO_{X_2}((l-1)\e12 + \ma E_2) \ot \pi^*\F \rightarrow \OO_{X_2}(l\e12 + \ma
E_2)\ot \pi^*\F \\
& \rightarrow {i_1}_*\OO_{\e12}(l\e12 + \ma E_2) \ot \pi^*\F.
\end{split}
\end{equation*}
We deduce a long sequence of triangles:

\begin{equation*}
\xymatrix@C=10pt{ 
\pi^*(\F)  \ar[rr]  & & F_{1}^{(2)}
\ar[ld] \ar[r] & \ldots F_{\ma}^{(2)}  \ar[rr] & & F_1^{(1)}
\ar[ld] \ar[r] & \ldots \ar[r] &  F_{3\ma}^{(1)} \ar[rr] & & F_{3\ma +1}^{(1)}
\ar[ld]  \\
  & \F_{1}^{(2)} \ar[lu] & \ldots  & &  \F_{1}^{(1)} \ar[lu] & \ldots &  &  &
\F_{3\ma + 1}^{(1)} \ar[lu] & \\
} 
\end{equation*}
where $F_k^{(2)} = \pi^*\F \ot \OO_{X_2}(kE_2)$, $\F_k^{(2)} = {i_2}_* \left(
\OO_{E_2}(kE_2) \ot i_2^* \pi^* \F \right)$, $F_l^{(1)} = \pi^*\F \ot
\OO_{X_2}(l\e12 + \ma E_2)$ and $\F_l^{(1)} = {i_1}_* \left( \OO_{\e12}(l\e12 +
\ma
E_2) \ot i_1^*\pi^*F \right)$. But we have commutative diagrams:

\begin{equation*}
\xymatrix{ 
E_2 \ar[rr]^{i_2} \ar[dd]^{\tilde{\pi_2}} & & X_2 \ar[dd]^{\pi_2}
\ar@/^6pc/[dddd]^{\pi} \\
& & \\
\pi_1^*\sigma_+(\GG) \ar[rr]^{j_2} & & X_1 \ar[dd]^{\pi_1} \\
& & \\
& & \tau(\GG)}
\end{equation*}
and

\begin{equation*}
\xymatrix{ 
\e12 \ar[rr]^{i_1} \ar[dd]^{\tilde{\pi}} & & X_2 \ar[dd]^{\pi}  \\
& & \\
\GG \ar[rr]^{j_1} & & \tau(\GG) \\}
\end{equation*}
so that 
\begin{equation*}
i_2^* \pi^* \F = \tilde{\pi_2}^* j_2^* \pi_1^*\F \subset \tilde{\pi_2}^*
\DP(\pi_1^*\sigma_+(\GG))
\end{equation*}
and 
\begin{equation*}
i_1^* \pi^*\F = \tilde{\pi}^* j_1^* \F \subset \tilde{\pi}^* \DP(\GG).
\end{equation*} 
Thus, $\pi^*\F$ is the $\D_{X_2}$-component of $F_{3\ma + 1}^{(1)} =
\pi^{!}(\F)$ in
the semi-orthogonal decomposition of proposition \ref{key}. As a consequence,
we have $\pi^*\F =
\delta \delta^{!}(\pi^{!}(\F))$, which is what we wanted to prove.

\end{proof}

\subsection{The key proposition}

In this section, we prove Proposition \ref{key}. We first recall the statement
of Proposition $4.1$ of \cite{kuz}):
\begin{prop}[Kuznetsov's Lefschetz decomposition] \label{lefschetzII}
Let $E$ be a Cartier divisor on a variety $X$. Assume that there is a
semi-orthogonal decomposition:
\begin{equation*}
\DB(E) = \langle \A_m \ot \OO_{E}(mE),\ldots, \A_1 \ot \OO_{E}(E),
\A_0
\rangle,
\end{equation*}
with $\A_m \subset \ldots \subset \A_0$ admissible subcategories of $\DB(E)$.
Then
there is a semi-orthogonal decomposition:
\begin{equation*}
\DB(X) = \langle i_*(\A_m \ot \OO_E(mE)),\ldots, i_*(\A_1 \ot \OO_{E}(E)), \D
\rangle,
\end{equation*}
where $i : E \hookrightarrow X$ is the natural inclusion and $\D = \{ \F \in
\DB(X),\, i^*\F \in \A_0 \}$.
\end{prop}

This result will be very
useful to deduce semi-orthogonal decompositions on
$X_2$, starting from semi-orthogonal decompositions on $E_2$ and $E_1^{(2)}$. To
prove \ref{key}, we need the following lemma:

\begin{lem} \label{decompo}
We have the following semi-orthogonal decomposition:
\begin{equation*}
\DB(E_2) = \langle \left( T_{E_2} \ot \OO_{E_2}(\ma E_2) \right),\ldots, \left(
T_{E_2} \ot \OO_{E_2}(E_2) \right), \mathcal{E}_2 \rangle,
\end{equation*}
where $T_{E_2} =
\tilde{\pi}_2^*\DB(\pi_1^*\sigma_+(\GG))$ and $\mathcal{E}_2$ is the left
orthogonal to the subcategory generated by the $T_{E_2} \ot \OO_{E_2}(k E_2)$,
for $1 \leq k \leq \ma$. Moreover, we have the inclusion:
\begin{equation*}
\tilde{\pi}_2^*\DB(\pi_1^*\sigma_+(\GG)) \subset \mathcal{E}_2.
\end{equation*}

We also have the semi-orthogonal decomposition:
\begin{equation*}
\begin{split}
\DB(E_1^{(2)}) = &  \langle T_{E_1^{(2)}} \ot \OO_{E_1^{(2)}}((3\ma +1)E_1^{(2)}
+ \ma E_2),\ldots, T_{E_1^{(2)}} \ot \\
& \OO_{E_1^{(2)}}(E_1^{(2)} + \ma E_2), {i_{1,2}}_* \left( T_{E_{1,2}} \ot
\OO_{E_{1,2}}(\ma E_{1,2}) \right),\ldots,\\
& {i_{1,2}}_* \left( T_{E_{1,2}} \ot \OO_{E_{1,2}}(E_{1,2}) \right),
\mathcal{E}_1^2 \rangle,
\end{split}
\end{equation*}
with $T_{E_1^{(2)}} = \tilde{\pi}^*\DB(\GG)$,
$T_{E_{1,2}} = \tilde{\pi}_{1,2}^* \DB(\pi_1^*\sigma_+(\GG) \cap E_1)$ and
$\mathcal{E}_1^2$ is the left orthogonal to the subcategory generated by the
$T_{E_1^{(2)}} \ot \OO_{E_1^{(2)}}(kE_1^{(2)} + \ma E_2)$ and the ${i_{1,2}}_*
\left( T_{E_{1,2}} \ot \OO_{E_{1,2}}(l E_{1,2}) \right)$, for $1 \leq k \leq
3\ma +1$ and $1 \leq l \leq \ma$. Moreover we have the inclusion:

\begin{equation*}
T_{E_{1}^{(2)}} \subset \mathcal{E}_1^2.
\end{equation*}

\end{lem}

\begin{proof}
We start with the proof of the first point. By Proposition \ref{vanishingII}, we
have:
\begin{equation*}
 \tilde{\pi}_{2*} \Hh(\OO_{E_2}(kE_2), \OO_{E_2}(kE_2)) =
\OO_{\pi_1^*\sigma_+(\GG)},
\end{equation*}
for all $k \in \mathbb{Z}$. This implies that the subcategories
$\tilde{\pi}_2^*\DB(\pi_1^*\sigma_+(\GG)) \ot \OO_{E}(kE)$ are full admissible
subcategories of $\DB(E_2)$, for all $k \in \mathbb{Z}$. Let $1 \leq k < l \leq
\ma + 1$ be integers. We have:

\begin{equation*}
\begin{split}
  & \HH \left( \OO_{E}(kE) \ot \tilde{\pi}_2^* \DB(\pi_1^*\sigma_+(\GG)),
\OO_{E}(lE) \ot \tilde{\pi}_2^*\DB(\pi_1^*\sigma_+(\GG)) \right)\\
= & \HH \left( \tilde{\pi_2}^* \DB(\pi_1^*\sigma_+(\GG)), \OO_{E}((l-k)E)
\ot \tilde{\pi_2}^*\DB(\pi_1^*\sigma_+(\GG)) \right) \\
= & \HH \left( \DB(\pi_1^*\sigma_+(\GG)), {\tilde{\pi_2}}_*\left(
\OO_{E}((l-k)E) \right) \ot \DB(\pi_1^*\sigma_+(\GG)) \right)  \\
= & 0,
\end{split}
\end{equation*}
where the last equality holds by Proposition \ref{vanishingII} because $1 \leq
l-k \leq
\ma$. As a consequence, we have a semi-orthogonal decomposition:
\begin{equation*}
\DB(E_2) = \langle \left( T_{E_2} \ot \OO_{E_2}(\ma E_2) \right),\ldots, \left(
T_{E_2} \ot \OO_{E_2}(E_2) \right), \mathcal{E}_2 \rangle,
\end{equation*}
with $T_{E_2} = \tilde{\pi}_2^*\DB(\pi_1^*\sigma_+(\GG))$ and $\mathcal{E}_2$
is the left orthogonal to the admissible subcategory generated by the
$\OO_{E}(kE) \ot \DB(\pi_1^*\sigma_+(\GG))$ for $1 \leq k \leq \ma $. It only
remains to show that:

\begin{equation*} 
\tilde{\pi}_2^* \DB(\pi_1^*\sigma_+(\GG)) \subset \mathcal{E}_1^2.
\end{equation*}
Equivalently, we need to prove that
the subcategory $\tilde{\pi_2}^*\DB(\pi_1^*\sigma_+(\GG))$
is left orthogonal to the admissible subcategory generated by the $\OO_{E}(kE)
\ot \tilde{\pi}_2^*\DB(\pi_1^*\sigma_+(\GG))$ for $1 \leq k \leq \ma $. As
before, this is a consequence of Proposition \ref{vanishingII}.

\bigskip

For the second point, we first note that the same proof as for the first point
yields the following semi-orthogonal decomposition:
\begin{equation*}
\DB(E_{1,2}) = \langle T_{1,2} \ot \OO_{E_{1,2}}(\ma E_{1,2}), \ldots T_{1,2}
\ot \OO_{E_{1,2}}(E_{1,2}), \mathcal{E}_{1,2} \rangle,
\end{equation*}
with $T_{1,2} = \tilde{\pi}_{1,2}^*\DB(\pi_1^*\sigma_+(\GG) \cap E_1)$ and
$T_{1,2} \subset \mathcal{E}_{1,2}$. So by Proposition \ref{lefschetzII}, the
categories
\begin{equation*}
{i_{1,2}}_* \left( \OO_{E_{1,2}}(k E_{1,2}) \ot
\tilde{\pi}_{1,2}^*\DB(\pi_1^*\sigma_+(\GG) \cap E_1) \right),
\end{equation*}
are full admissible subcategories of $\DB(E_1^{(2)})$ which are left orthogonal
to each other, for $1 \leq k \leq \ma$. 

Using again Proposition \ref{vanishingII}, we prove that
the subcategories:
\begin{equation*}
\tilde{\pi}^*\DB(\GG) \ot \OO_{E_1^{(2)}}(lE_1^{(2)} + \ma E_2)
\end{equation*}
are full admissible subcategories of $\DB(E_1^{(2)})$ which are left orthogonal
to each other, for $1 \leq l \leq 3 \ma +1$. The adjunction formula shows that
$\w_{E_1^{(2)}} = \OO_{E_1^{(2)}}( \ma E_2) \ot \mathcal{w}$ for some
$\mathcal{w} \in T_{E_1^{(2)}}$. Then, by Serre duality, we have:

\begin{equation*}
\begin{split}
  & \HH \left( {i_{1,2}}_* \left( T_{E_{1,2}} \ot \OO_{E_{1,2}}(kE_{1,2})
\right), T_{E_1^{(2)}} \ot \OO_{E_1^{(2)}}(lE_1^{(2)} + \ma E_2)
 \right) \\
 = & \HH \left( T_{E_1^{(2)}} \ot \OO_{E_1^{(2)}}(lE_1^{(2)}), {i_{1,2}}_*
\left( T_{E_{1,2}} \ot \OO_{E_{1,2}}(kE_{1,2}) \right) \right)^*,
\end{split}
\end{equation*}
for $1 \leq k \leq \ma$ and $1\leq l \leq 3\ma +1$. Recall that
$T_{E_1^{(2)}} = \pi_{1,2}^* T_{E_1}$, with $T_{E_1} = \tilde{\pi}_1^*\DB(\GG)$
and that $\OO_{E_1^{(2)}}(E_1^{(2)}) = \pi_{1,2}^* \OO_{E_1}(E_1)$. Thus, the
adjunction formula for $\pi_{1,2}$ gives:

\begin{equation*}
\begin{split}
  & \HH \left( T_{E_1^{(2)}} \ot \OO_{E_1^{(2)}}(lE_1^{(2)}), {i_{1,2}}_* \left(
T_{E_{1,2}} \ot \OO_{E_{1,2}}(kE_{1,2}) \right) \right) \\
= & \HH \left( T_{E_1} \ot \OO_{E_1}(lE_1), \pi_{1,2*} \left( {i_{1,2}}_*
\left( T_{E_{1,2}} \ot \OO_{E_{1,2}}(kE_{1,2}) \right) \right) \right).
\end{split}
\end{equation*}

But we have a commutative diagram:

\begin{equation*} 
\xymatrix{ 
E_{1,2} \ar@{^{(}->}[rrr]^{i_{1,2}} \ar[dd]^{\tilde{\pi}_{1,2}} & & & \e12
\ar[dd]^{\pi_{1,2}}  \\
& & \\
E_1 \cap \pi_1^*\sigma_+(\GG) \ar@{^{(}->}[rrr]^{\kappa_1} & & & E_1 \\}
\end{equation*}

so that:
\begin{equation*}
{\pi_{1,2}}_* \left( {i_{1,2}}_* \left( T_{E_{1,2}} \ot \OO_{E_{1,2}}(kE_{1,2})
\right) \right) = {\kappa_1}_* \left( \tilde{\pi}_{1,2*} \left( T_{E_{1,2}}
\ot \OO_{E_{1,2}}(kE_{1,2}) \right) \right)
\end{equation*}
and by Proposition \ref{vanishingII}, we have $\tilde{\pi}_{1,2*} \left(
\OO_{E_{1,2}}(kE_{1,2}) \right) = 0$ for all $1 \leq k \leq \ma$. 

\bigskip

As a consequence, we have proved that we have a semi-orthogonal decomposition:
\begin{equation*}
\begin{split}
\DB(E_1^{(2)}) = &  \langle T_{E_1^{(2)}} \ot \OO_{E_1^{(2)}}((3\ma +1)E_1^{(2)}
+ \ma E_2),\ldots, T_{E_1^{(2)}} \ot \OO_{E_1^{(2)}}(E_1^{(2)} + \ma E_2),\\
 & {i_{1,2}}_* \left( T_{E_{1,2}} \ot \OO_{E_{1,2}}(\ma E_{1,2}) \right),\ldots,
{i_{1,2}}_* \left( T_{E_{1,2}} \ot \OO_{E_{1,2}}(E_{1,2}) \right),
\mathcal{E}_1^2 \rangle,
\end{split}
\end{equation*}
with $T_{E_1^{(2)}} = \tilde{\pi}^*\DB(\GG)$, $T_{E_{1,2}} =
\tilde{\pi}_{1,2}^*
\DB(\pi_1^*\sigma_+(\GG) \cap E_1)$ and $\mathcal{E}_1^2$ is the left orthogonal
to the admissible subcategory generated by the $T_{E_1^{(2)}} \ot
\OO_{E_1^{(2)}}(lE_1^{(2)} + \ma E_2)$ and ${i_{1,2}}_* \left( T_{E_{1,2}} \ot
\OO_{E_{1,2}}(k E_{1,2}) \right)$ for $1 \leq l \leq 3 \ma + 1$ and $1 \leq k
\leq \ma$.

\bigskip

It remains to prove that $T_{\e12} \subset \mathcal{E}_1^{(2)}$. This is done as
before using Proposition \ref{vanishingII}. We leave the proof to the reader.

\end{proof}

Using this lemma, we can finish the proof of proposition \ref{key}. 

\begin{proof}[Proof of Proposition \ref{key}]

By Proposition \ref{lefschetzII}, we know that the categories  
\begin{equation*}
\A_l = {i_1}_* \left( \tilde{\pi}^*\DB(\GG) \ot \OO_{\e12}(l \e12 + \ma E_2)
\right)
\end{equation*}
and
\begin{equation*}
\B_k = {i_2}_* \left( \tilde{\pi}_2^*\DB(\pi_1^*\sigma_+(\GG)) \ot
\OO_{E_2}(kE_2) \right)
\end{equation*}
are full admissible subcategories of $\DB(X_2)$ for $1 \leq l \leq 3 \ma +1$
and $1 \leq k \leq \ma$. Moreover, again by Proposition \ref{lefschetzII}, the
$\A_l$ are left orthogonal to each other for $1 \leq l \leq 3 \ma +1$, while the
$\B_k$ are left orthogonal to each other for $1 \leq k \leq \ma$. We start by
proving that the $\B_k$ are left orthogonal to the $\A_l$. We have:

\begin{equation*}
\begin{split}
  & \HH(\B_k, \A_l)\\
= & \HH ( {i_2}_* \left( \tilde{\pi}_2^*\DB(\pi_1^*\sigma_+(\GG)) \ot
\OO_{E_2}(kE_2) \right),\\
  & \,\,\,\,\,\,\,\,\,\, {i_1}_* \left( \tilde{\pi}^*\DB(\GG) \ot \OO_{\e12}(l
\e12 + \ma E_2)
\right))\\
= & \HH ( i_1^* \left[  {i_2}_* \left( \tilde{\pi_2}^*\DB(\pi_1^*\sigma_+(\GG))
\ot \OO_{E_2}(kE_2) \right) \right], \\
  & \,\,\,\,\,\,\,\,\,\, \tilde{\pi}^*\DB(\GG) \ot \OO_{\e12}(l \e12 + \ma
E_2)).
\end{split}
\end{equation*}

But we have a cartesian square:

\begin{equation*}
\xymatrix{ 
E_{1,2} \ar[rr]^{i_{1,2}} \ar[dd]^{i_{2,1}} & &  \e12 \ar[dd]^{i_1}  \\
& & \\
E_2 \ar[rr]^{i_2} & &  X_2 \\},
\end{equation*}
with $\dim E_{1,2} = \dim E_2 + \dim \e12 - \dim X_2$ and $i_1, i_2$ are locally
complete intersection embeddings. Thus, this diagram is Tor-neutral (see
\cite{kuz4}, Corollary $2.27$) and we have:
\begin{equation*}
i_1^* {i_2}_* \F = {i_{1,2}}_* i_{2,1}^* \F,
\end{equation*}
for all $\F \in \DM(E_2)$. So we have:
\begin{equation*}
\begin{split}
  & i_1^* \left[  {i_2}_* \left( \tilde{\pi}_2^*\DB(\pi_1^*\sigma_+(\GG)) \ot
\OO_{E_2}(kE_2) \right) \right] \\
= & {i_{1,2}}_* \left[ i_{2,1}^* \left(
\tilde{\pi}_2^*\DB(\pi_1^*\sigma_+(\GG))
\ot \OO_{E_2}(kE_2) \right) \right].
\end{split}
\end{equation*}
The commutative diagram:

\begin{equation*}
\xymatrix{ 
E_{1,2} \ar@{^{(}->}[rrr]^{i_{2,1}} \ar[dd]^{\tilde{\pi_{2,1}}} & &  & E_2
\ar[dd]^{\tilde{\pi_2}}  \\
& & \\
E_1 \cap \pi_1^*\sigma_+(\GG) \ar@{^{(}->}[rrr]^{j_2} & &  &
\pi_1^*\sigma_+(\GG) \\},
\end{equation*}
shows that:
\begin{equation*}
i_{2,1}^* \left( \tilde{\pi}_2^*\DB(\pi_1^*\sigma_+(\GG)) \right) =
\tilde{\pi_{2,1}}^* \left( {j_2}_* \DB(\pi_1^*\sigma_+(\GG)) \right).
\end{equation*}
As $\pi_1^*\sigma_+(\GG)$ is smooth, we have the inclusion: 
\begin{equation*}
\tilde{\pi}_{1,2}^* \left( {j_2}_* \DB(\pi_1^*\sigma_+(\GG)) \right) \subset
T_{E_{1,2}}.
\end{equation*}
Hence, to prove that $\HH(\B_k, \A_l) = 0$, it is sufficient to prove that:
\begin{equation*}
\HH ( {i_{1,2}}_* \left(T_{E_{1,2}} \ot \OO_{\e12}(kE_2) \right), T_{\e12} \ot
\OO_{\e12}(l \e12 + \ma E_2)) = 0.
\end{equation*}
Since $\OO_{ \e12 }(E_2) = \OO_{\e12 }(E_{1,2})$, this last vanishing
comes
from the semi-orthogonal decomposition of $\DB(E_1)$ found in Lemma
\ref{decompo}. 
As a consequence, we have proved that we have a semi-orthogonal decomposition:

\begin{equation*}
\DB(X_2) = \langle \A_{3 \ma +1},\ldots, \A_1, \B_{\ma},\ldots, \B_1, \D_{X_2}
\rangle,
\end{equation*}
where $\D_{X_2}$ is the left orthogonal to the admissible subcategory generated
by the $\A_l$'s and the $\B_k$'s. The only fact left to complete the proof of
Proposition \ref{key} is the inclusion $\pi^* \DP(\tg) \subset \D_{X_2}$.
Equivalently, we have to prove that $\pi^* \DP(\tg)$ is left orthogonal to the
$\A_l$'s and the $\B_k$'s. By Proposition \ref{lefschetzII}, the left orthogonal
to
the $\B_k$'s is:
\begin{equation*}
\{ \F \in \DB(X_2),\, i_2^* \F \in \mathcal{E}_2 \}
\end{equation*}
 and we know (Proposition \ref{decompo}) that $\tilde{\pi}_2^* \DB(\pi_1^*
\sigma_+(\GG)) \subset \mathcal{E}_2$. Moreover the commutative diagram:

\begin{equation*}
\xymatrix{ 
E_2 \ar[rr]^{i_2} \ar[dd]^{\tilde{\pi}_2} & & X_2 \ar[dd]^{\pi_2}
\ar@/^6pc/[dddd]^{\pi} \\
& & \\
\pi_1^*\sigma_+(\GG) \ar[rr]^{j_2} & & X_1 \ar[dd]^{\pi_1} \\
& & \\
& & \tau(\GG)}
\end{equation*}
shows that $i_2^* \pi^* \DP(\tg) \subset \tilde{\pi}_2^*
\DB(\pi_1^*\sigma_+(\GG))$, which implies that $\pi^* \DP(\tg)$ is left
orthogonal to the $\B_k$'s, for $1 \leq k \leq \ma$. We prove in the same
fashion
that $\pi^* \DP(\tg)$ is left orthogonal to the $\A_l$'s for $1 \leq l \leq 3
\ma +1$. This concludes the proof of Proposition \ref{key}.

\end{proof}

\newpage

\section{Conclusion}
In this paper, we developed methods in order to build categorical crepant
resolutions of singularities for some varieties which do not admit any wonderful
resolution of singularities. In \cite{abuafcategorical}, we noticed that
strongly crepant categorical resolutions have very interesting minimality
properties. Unfortunately, they seem to be much more difficult to construct. As
a corollary of the main result of \cite{abuafcategorical}, we know that all
determinantal varieties admit categorical crepant resolutions of singularities.
Which determinantal varieties have a strongly crepant resolution is yet a widely
open problem:

\begin{quest} Which determinantal varieties admit strongly crepant categorical resolution of singularities?
\end{quest}
From \cite{acgh}, section $2.2$, we know that all square determinantal
varieties have a small resolution of singularities, hence a (geometric) crepant
resolution. Thus, the above question is only interesting for symmetric and
Pfaffian determinantal varieties. In the Appendix B, we will prove that all
Pfaffians are $\mathbb{Q}$-factorial with terminal singularities, so that they
do not admit any geometric crepant resolution of singularities.

\bigskip

Let us mention some obstructions to the construction of strongly crepant categorical resolution of singularities. Let $X$ be a projective variety with Gorenstein rational singularities and let:
\begin{equation*}
\pt : \T \rightarrow \DB(X), 
\end{equation*}
a categorical crepant resolution of $X$. If $S_{\T}$ is a Serre functor for $\T$,
then, for all $\F \in \DP(X)$, we have:
\begin{equation*}
S_{\T}(\pi_{\T}^* \F) = \pi_{\T}^* \F \ot \pi_{\T}^* \w_{X}[\dim X],
\end{equation*}
where $\pt^*$ is the left adjoint to ${\pt}_*$.
In order for $\T$ to be a strongly crepant resolution, we need:
\begin{equation*}
 S_{\T}(T) = T \ot \pi_{\T}^*\w_{X}[\dim X],
\end{equation*}
for all $T \in \T$. If $\T \simeq \DB(Y)$ for some variety $Y$, then the Serre
functor of $\T$ is the tensor product with the dualizing complex of $Y$. As $\T$
is a crepant resolution of $X$, we deduce that the dualizing complex of $Y$ is
$\pi_{\T}^*
\w_{X}[
\dim X]$. Nevertheless, if $\T$ is not \textit{geometric}, then we can not
predict how the Serre functor of $\T$ acts on objects which are not in
$\pi_{\T}^* \DP(X)$. Kuznetsov gives examples of categorical crepant resolutions
which are not strongly crepant (\cite{kuz}). In particular, the Serre functor
does not act the same on all objects of these categories. In order to construct
strongly crepant resolutions of singularities, it seems necessary to understand
in details the objects $T \in \T$ such that $S_{\T}(T) \neq T \ot \pi_{\T}^*
\w_X[\dim X]$. This more or less reduces to understand more precisely
categorical crepant resolutions of singularities. A complete answer to the
following question could prove to be very helpful:

\begin{quest} Let $X$ be a projective variety with Gorenstein rational singularities. let $\pi :\tilde{X}  \rightarrow X$ be a resolution of singularities and:
\begin{equation*}
 \xymatrix{ \T \ar[dd]^{{\pi_{\T}}_*} \ar@{^{(}->}[rr]^{\d} & & \DB(\tilde{X})
\ar[lldd]_{\pi_*}\\
 & & \\
 \DB(X)}
\end{equation*}
be a categorical crepant resolution of $X$. When is $\T$ the derived category of a (non-commutative) moduli space of objects in $\DB(\tilde{X})$ ? 
\end{quest}
Note that the idea of linking a non trivial component of a Lefschetz
decomposition of $\DB(Y)$ (for smooth $Y$) to a moduli space of objects in
$\DB(Y)$ is not new
(see \cite{kuz8}, \cite{kuz9}). It had been fruitfully exploited in \cite{macri-stellari}.

\bigskip

Let us come back to the case of $\mathrm{G(3,6)} \subset \mathbb{P}^{19}$. We prove that the tangent variety of $\mathrm{G(3,6)}$ admit a categorical crepant resolution of singularities. One would like to know if this resolution is \textit{non-commutative}, more precisely:

\begin{quest} Let 
\begin{equation*}
\xymatrix{
\D_{X_2} \ar[ddrr]^{\pt} \ar@{^{(}->}[rr]^{\d} & & \DB(X_2)
\ar[dd]^{\pi_*} \\
& & \\
& & \DB(\tau(\mathrm{G(3,6)})) \\}
\end{equation*}
be the categorical crepant resolution of the tangent variety to $\mathrm{G(3,6)}$ built in the theorem \ref{mainII}. Is there a sheaf of algebras \footnote{A daring mind would not restrict to the sole algebras, but would also consider \textit{DG-algebras} and perhaps
\textit{$A_{\infty}$-algebras...}} $\A_{\tau(\mathrm{G(3,6)})}$ on
$\tau(\mathrm{G(3,6)})$ such that:
\begin{equation*}
\D_{X_2} \simeq
\DB(\tau(\mathrm{G(3,6)}), \A_{\tau(\mathrm{G(3,6)})}) \,\,\, ?
\end{equation*}
\end{quest}
As a consequence of theorem $5.2$ in \cite{kuz}, it is sufficient to find "very
good" semi-orthogonal decompositions of the exceptional divisors of $X_2$. In
our situation, it would be sufficient to find a pair of vector bundles $\{
\mathcal{V}, \mathcal{W} \}$ on $X_2$, which is exceptional with respect
to $\pi$, such that:
\begin{equation*}
 \mathcal{E}_2 = \langle \mathcal{V}|_{E_2} \ot \tilde{\pi}_2^* \DB(\pi_1^*
\sigma_+(\mathrm{G(3,6)})) \rangle
\end{equation*}
and
\begin{equation*}
\mathcal{E}_1^2 = \langle \mathcal{W}|_{E_1^{(2)}} \ot \tilde{\pi}^*
\DB(\mathrm{G(3,6)}) \rangle,
\end{equation*}
where $\mathcal{E}_2$ and $\mathcal{E}_1^2$ are defined in the lemma
\ref{decompo}.
\bigskip

The natural projection $\tilde{\pi_2} : E_2 \rightarrow \pi_1^*
\sigma_+(\mathrm{G(3,6)}))$ is a fibration into smooth quadrics. It seems
plausible to find two vector bundles on $X_2$ which would specialize into the
relative spinor bundles once restricted to $E_2$. This would rule out the case
of $\mathcal{E}_2$. The case of $\E_1^{(2)}$ is much more subtle. Indeed, the
projection $\tilde{\pi}_1 : E_1 \rightarrow \mathrm{G(3,6)}$ is
a fibration into doubled $\mathbb{P}^8$. As a consequence, the divisor $E_1$ is globally non reduced and so is the divisor $E_1^{(2)}$. The existence of a vector bundle $\mathcal{W}$ on $X_2$ such that:
\begin{equation*}
 \mathcal{E}_1^{(2)} = \langle \mathcal{W}|_{E_2} \ot \tilde{\pi}_2^*
\DB(\pi_1^*
\sigma_+(\mathrm{G(3,6)})) \rangle
\end{equation*}
would imply that $\DB(E_1^{(2)})$ has finite homological dimension : it is
impossible! One can however hope that the theorem $5.2$ of \cite{kuz} could be
extended in the following way : it is sufficient to find a sheaf of algebras
$\A_{X_2}$ on $X_2$ with finite homological dimension, such that $\DB(E_1^{(2)},
\A_{X_2} \ot \OO_{E_1^{(2)}})$ still has a "very good" semi-orthogonal decomposition and the natural projection:                                                           
\begin{equation*}
\DB(E_1^{(2)}, \A_{X_2} \ot \OO_{E_1^{(2)}}) \rightarrow \DB(E_1^{(2)}),
\end{equation*}
is a \textit{categorical resolution of singularities} (in some extended sense for non-reduced schemes). Fortunately enough, part of this program has been already carried out in \cite{kuz-lunts}. Indeed, theorem $5.23$ of this paper enables us to construct a sheaf of algebras $\A_{E_1}$ on
$E_1$ such that:
\begin{itemize}
 \item the natural projection $\DB(E_1, \A_{E_1}) \rightarrow
\DB(E_1)$ is a \textit{categorical resolution of singularities},
\item there is a semi-orthogonal decomposition:
\begin{equation*}
\begin{split}
\DB(E_1, \A_{E_1}) = & \langle \langle
\OO_{|E_1|_{red}}^{\alpha}(-8) \ot
\tilde{\pi}_1^* \DB(\mathrm{G(3,6)}), \ldots 
\OO_{|E_1|_{red}}^{\alpha} \ot
\tilde{\pi}_1^* \DB(\mathrm{G(3,6)} \rangle, \\
& \langle \OO_{|E_1|_{red}}^{\beta}(-8) \ot
\tilde{\pi}_1^* \DB(\mathrm{G(3,6)}), \ldots 
\OO_{|E_1|_{red}}^{\beta} \ot
\tilde{\pi}_1^* \DB(\mathrm{G(3,6)} \rangle \rangle, 
\end{split}
\end{equation*}
where $\OO_{[E_1[_{red}}^{\alpha}(1)$ and $\OO_{|E_1|_{red}}^{\beta}(1)$ are
sheaves of $\A_{E_1}$-modules which identify to the relatively very ample
generator of the relative Picard group of the projective bundle: 
\begin{equation*}
|E_1|_{red} \rightarrow \mathrm{G(3,6)},
\end{equation*}
when they are restricted to $|E_1|_{red}$, the reduced scheme underlying $E_1$.
\end{itemize}

\bigskip

Finally, there are many points left to check in order to demonstrate the existence of a non-commutative crepant resolution of $\tau(\mathrm{G(3,6)})$:
\begin{itemize}
 \item find a sheaf of algebras $\A_{E_1^{(2)}}$ on $E_1^{(2)}$ such that
$\DB(E_1^{(2)}, \A_{E_1^{(2)}})$ admits a semi-orthogonal decomposition compatible with the one of
$\DB(E_1, \A_{E_1})$,
\item show that $\A_{E_1^{(2)}}$ is the restriction to 
$E_1^{(2)}$ of a sheaf of algebras $\A_{X_2}$ on $X_2$ (this should be the trickiest part!),
\item show that the natural projection:
\begin{equation*}
r_* : \DB(X_2,\A_{X_2}) \rightarrow \DB(X_2),
\end{equation*}
satisfy $r_* r^* = \mathrm{id}$,
\item prove that the category $\DB(E_2, \A_{X_2} \ot \OO_{E_2})$ still has a "very good" semi-orthogonal decomposition which is compatible with the decomposition of $\DB(E_1^{(2)}, \A_{E_1^{(2)}})$.

\end{itemize}
The first point of this program is the certainly the easiest to complete. Indeed, the divisor
$E_1^{(2)}$ is the bow-up of $E_1$ along a smooth subscheme which meet transversally all fibers of $\tilde{\pi}_1 : E_1
\rightarrow \mathrm{G(3,6)}$.

\newpage

\begin{appendix}
\section{The map $\mu : E \rightarrow \sigma_+(\GG)$ has infinite
Tor-dimension}

\subsection{Basic facts on finite Tor-dimension}
We recall the following definition:

\begin{defi}
Let $f : X \rightarrow Y$ be a morphism of schemes of finite type over an
algebraically closed field $k$. We say that $f$ has finite \emph{Tor-dimension}
if $\OO_{X}$ has a finite projective resolution as a $\OO_Y$-module.
\end{defi}

The following are the best known examples of morphism with finite Tor-dimension:
\begin{itemize}
\item morphisms $f : X \rightarrow Y$, with $Y$ smooth,
\item flat morphisms,
\item locally complete intersection morphisms,
\item any composition of the three above examples.
\end{itemize}

The result below implies that any resolution of singularities has infinite
Tor-dimension:

\begin{prop} \label{easyinf} Let $f : X  \rightarrow Y$ a proper morphism of
varieties over an algebraically closed field $k$. Let $y \in Y_{sing}$ and
assume that $f^{-1}(y)$ is not included in the singular locus of $X$. Then $f$
has
infinite Tor-dimension.
\end{prop}
As I was not able to find any proper reference for this standard fact, I provide
a proof of it.

\begin{proof} Since $f$ takes closed point to closed points, this question can
be localized
at the neighborhood of any point in $X$. Thus, we have to prove the following:
Let $f : A \rightarrow B$ be a morphism of local Noetherian rings whose residue
fields are $k$ and with $B$ regular. Assume that $f$ has finite Tor-dimension.
Then $A$ is also regular.

\bigskip

We first consider a finite free resolution of $k$ as a $B$-module:
\begin{equation*}
0 \rightarrow M_r \rightarrow \cdots \rightarrow M_p \rightarrow \cdots
\rightarrow M_0
\rightarrow k \rightarrow 0.
\end{equation*} 
Then let:
\begin{equation*}
\cdots \rightarrow N_{q,p} \rightarrow \cdots \rightarrow N_{0,p} \rightarrow
M_p \rightarrow 0
\end{equation*}
be a (possibly infinite) resolution of $M_p$ by free $A$-modules. Since
all $N_{q,p}$ are free $A$-modules, the map $M_p \rightarrow M_{p-1}$
lifts to a map $N_{q,p} \rightarrow N_{q,p-1}$ for all $0 \leq p \leq r$ and $q
\geq 0$. Thus, we get an infinite double complex of free $A$-modules whose
terms are the $N_{q,p}$ for $0 \leq p \leq r$ and $q \geq 0$. 

The ring $B$ has finite Tor-dimension (say $t$) on $A$, so that the
kernel
$K_{t+1,p}$ of $N_{t,p} \rightarrow N_{t-1,p}$ is flat for all $0 \leq p \leq
r$. Since all squares appearing in the double complex $N_{p,q}$ commute, we can
lift the map $N_{t,p} \rightarrow N_{t,p-1}$ to a map $K_{t+1,p} \rightarrow
K_{t+1,p-1}$. As a consequence, we get a finite double complex $G_{\bullet,
\bullet}$ of flat $A$-modules, where $G_{q,p} = N_{q,p}$ for $0 \leq q \leq t$,
$G_{t+1,p} = K_{t+1,p}$ and $G_{q,p} = 0$ for $q>t+1$.

By the Cartan-Eilenberg resolution, the simple complex associated to the double
complex $G_{\bullet,\bullet}$ is quasi-isomorphic to the complex $M_{\bullet}$.
Hence $k$ admits a finite resolution by flat $A$-modules, so that $A$ is
regular (see \cite{matsumura}, Theorem $19.2$).

\end{proof}
\subsection{Growth of infinite free resolutions}

In this section, we come back to the case of the morphism $\mu : E \rightarrow
\sigma_+(\GG)$ and we prove the following:
\begin{prop}
The morphism $\mu : E \rightarrow \sigma_+(\GG)$ has infinite Tor-dimension.
\end{prop}

Note that this result is not completely obvious. Indeed, the singular locus of
$E$ is precisely the inverse image by $\mu$ of the singular locus of
$\sigma_{+}(\GG)$, so that we cannot apply Proposition \ref{easyinf} to prove
the statement. However we will stick to the following principle:
\begin{center}
\emph{If $f : X \rightarrow Y$ has finite Tor-dimension, then the singularities
of $Y$ can't be much worse than the singularities of $X$.}
\end{center}

We will make this idea precise using the theory of growth of Betti numbers for
infinite free resolutions. We refer to \cite{avramov} for a nice exposition of
this theory.

\begin{defi}
Let $B$ be a local noetherian ring with residue field $k$, an algebraically
closed field of char $0$. Let $\F$ be a module of finite type on $B$ and let:
\begin{equation*}
\cdots \rightarrow M_n \rightarrow M_{n-1}\rightarrow \cdots \rightarrow M_1
\rightarrow \F \rightarrow 0,
\end{equation*}
be a (possibly infinite) minimal resolution of $\F$ by free $B$-modules.
The \emph{$n$-th Betti number} of $\F$, which we denote by $\beta^n(\F)$, is the
rank of $M_n$.
\end{defi}
Note that $\beta^n(\F)$ is also equal to the dimension of $Tor^n_{B}(\F,k)$. 

\begin{defi}
With the same hypothesis as above, we define the \emph{complexity} of $\F$ to
be:
\begin{equation*}
\mathrm{cp}(\F) = \mathrm{min}\{d,\, \exists \alpha \in \mathbb{R} \,\,
\text{such that}\, \beta^n(\F) \leq \alpha.n^{d-1}\, \text{for all}\,\, n>>0 \}.
\end{equation*}
\end{defi}

The following result characterizes locally complete intersection in terms of
complexity (see \cite{avramov} remark $8.1.3$).

\begin{theo} \label{complexity}
Let $B$ be a local notherian ring whose residue field is $k$. Assume that
$\mathrm{cp}(k) < + \infty$. Then $B$ is a complete intersection in a regular
local ring. Moreover, assume that $B$ is Cohen-Macaulay and that $\mathrm{cp}(k)
\leq 1$, then $B$ is a hypersurface ring in a regular local ring.
\end{theo}

The converse of the above theorem holds and is much easier. In the case of
hypersurfaces, there is even a more precise result. We start with a definition:

\begin{defi}
Let $B$ be a notherian local ring with residue field $k$. Let $\F_{\bullet}$ be
an unbounded from below complex of modules over $B$. We say that $\F_{\bullet}$
is \emph{periodic at infinity of period
$p>0$}, if there exists an integer $m$ such that for all $i<m$, we have: 
\begin{equation*}
\F_{i-p} = \F_{i}
\end{equation*}
and a commutative diagram:
\begin{equation*}
\xymatrix{
\F_{i-p+1} \ar[rr]^{\partial_{i-p+1}} \ar@{=}[dd] & & \F_{i-p} \ar@{=}[dd] \\
& & \\
\F_{i+1} \ar[rr]^{\partial_{i+1}} & & \F_{i} \\}
\end{equation*}
\end{defi}
The following is one of the fundamental results in the theory of matrix
factorizations (see \cite{avramov}, construction $5.1.2$):

\begin{theo} \label{matrixfact}
Let $B$ be a notherian local ring which is a hypersurface ring in some local
regular ring. Then, any module of finite type $\F$ over $B$ admits a resolution
by a complex $\tilde{\F}^{\bullet}$ of finite free $B$-modules, periodic at
infinity of period $2$.
\end{theo}

We can now prove the main result of this section:

\begin{theo} \label{hardinfinite}
Let $f : A \rightarrow B$ be a local morphism of noetherian local rings with
residue field $k$ and with $A$ Cohen-Macaulay. Assume that $B$ is a hypersurface
ring in a regular local ring and that $f$ has finite Tor-dimension. Then $A$ is
also a hypersurface ring in a regular local ring.
\end{theo}
This result (and its proof) is somehow similar to its analogue \ref{easyinf}.

\begin{proof}
We start with a periodic resolution of $k$ by finite free $B$-modules:
\begin{equation*}
\cdots M_p \stackrel{\partial_p^B}\rightarrow M_{p-1} \rightarrow \cdots
\stackrel{\partial_1^B}\rightarrow M_0 \rightarrow k \rightarrow 0,
\end{equation*}

with

\begin{equation*}
M_{i-2} = M_{i}
\end{equation*}
and a commutative diagram:
\begin{equation*}
\xymatrix{
M_{i-1} \ar[rr]^{\partial_{i-1}^M} \ar@{=}[dd] & & M_{i-2} \ar@{=}[dd] \\
& & \\
M_{i+1} \ar[rr]^{\partial_{i+1}^M} & & M_{i} \\}
\end{equation*}
for all $i \ll 0$.

\bigskip

Since $f$ has finite Tor-dimension, the same argument as in the proof of
proposition \ref{easyinf} shows that we can find a double complex $N_{\bullet,
\bullet}$ of flat $A$-modules, such that $N_{\bullet,p}$ is a finite resolution
of $M_p$ by flat $A$-modules. Since the complex $M_{\bullet}$ is periodic at
infinity of period $2$, we get:

\begin{equation*}
N_{\bullet,i-2} = N_{\bullet,i}
\end{equation*}
and
\begin{equation*}
\xymatrix{
N_{\bullet,i-1} \ar[rr]^{\partial_{\bullet,i-1}^N} \ar@{=}[dd] & &
N_{\bullet,i-2} \ar@{=}[dd] \\
& & \\
N_{\bullet,i+1} \ar[rr]^{\partial_{\bullet,i+1}^N} & & N_{\bullet,i} \\}
\end{equation*}
for $i<<0$. Let $G_{\bullet}$ be the Cartan-Eilenberg resolution of $N_{\bullet,
\bullet}$. This is an unbounded from below, periodic at infinity, complex of
flat
$A$-modules which is quasi-isomorphic to $k$. Since all the $G_q$ are flat
$A$-modules, the $Tor^{n}_{A}(k,k)$ are the homology groups of the complex
$G_{\bullet} \ot_{A} k$. But the very definition of periodicity at infinity
implies that the sequence of homology groups $\mathcal{H}_i(G_{\bullet} \ot_{A}
M)$ is periodic for all $A$-modules $M$ and $i \ll 0$. As a consequence, the
$Tor^{n}_A(k,k)$ are periodic for $n>>0$. But the ring $A$ is Cohen-Macaulay, so
by Proposition \ref{complexity}, the ring $A$ is a hypersurface in a regular
local ring.
\end{proof}

Now we can prove that the map $\mu : E \rightarrow \GG$ has infinite
Tor-dimension. We proceed by contradiction. Assume that $\mu$ has finite
Tor-dimension. Since $E$ is a Cartier divisor in a smooth variety and
$\sigma_+(\GG)$ is Cohen-Macaulay, we can apply Theorem \ref{hardinfinite} and
we find that for any $x \in \sigma_+(\GG)$, there exists an open subset $U_x$ of
$\sigma_+(\GG)$ containing $x$ such that $U_x$ is a hypersurface in a smooth
scheme, say $V_x$. Let $\kappa : \tilde{V_x} \rightarrow V_x$ be the blow-up of
$V_x$ along $U_x \cap \GG$ and denote by $E_V$ the exceptional divisor. The
strict transform of $U_x$ by $\kappa$ is the blow-up of $U_x$ along $U_x \cap
\GG$, whose exceptional divisor $E_U$ is a fibration into $\AP$ over $U_x$ (see
Proposition \ref{desingsigma}). Since $U_x$ is a hypersurface in $V_x$, the
fibers of $E_U$ over $\GG$ are hypersurfaces in the fibers of $E_V$ over $\GG$.
As a consequence $\AP$ is a hypersurface in some projective space. We will show
that it is impossible.

Indeed, let us first consider the case $\AA = \mathbb{R}$. Then $\AP =
\mathbb{P}^2$. All embeddings of $\mathbb{P}^2$ in projective spaces are given
by powers of
$\OO_{\mathbb{P}^2}(1)$ followed by linear projections. The only embedding of
$\mathbb{P}^2$ as a hypersurface is thus the embedding in $\mathbb{P}^3$ as a
hyperplane. But
if the tangent cone of $U_x$ at $y \in \GG$ is a hyperplane in the
tangent space to $V_x$ at $y$, then $U_x$ is smooth at $y$ which is a
contradiction.

For $\AA = \mathbb{C}$, $\mathbb{H}$, or $\mathbb{O}$, we use a topological
argument to get a contradiction. We recall that in the cases $\AA =
\mathbb{C}$, $\mathbb{H}$, and $\mathbb{O}$, the Severi varieties are
$\mathbb{P}^2 \times \mathbb{P}^2$, $\mathrm{Gr}(2,6)$ and $\mathbb{OP}^2$. The
first integer cohomology groups of these varieties are described in the
following table:
\begin{equation*}
\begin{array}{|c|c|c|c|c|}
\AP & \dim \AP & H^0(\AP, \mathbb{Z}) & H^2(\AP, \mathbb{Z}) & H^4(\AP,
\mathbb{Z}) \\
\hline
\mathbb{P}^2 \times \mathbb{P}^2 & 4 & \mathbb{Z} & \mathbb{Z} \oplus \mathbb{Z}
& \mathbb{Z} \oplus \mathbb{Z} \oplus \mathbb{Z}\\

\mathrm{Gr}(2,6) & 8 &  \mathbb{Z} &  \mathbb{Z} & \mathbb{Z} \oplus
\mathbb{Z}\\
\mathbb{OP}^2 & 16 & \mathbb{Z} & \mathbb{Z} & \mathbb{Z} \\
\hline
\hline
H^6(\AP, \mathbb{Z}) & H^8(\AP, \mathbb{Z}) & H^{10}(\AP,
\mathbb{Z}) & H^{12}(\AP, \mathbb{Z}) & H^{14}(\AP, \mathbb{Z}) \\
\hline
\mathbb{Z} \oplus \mathbb{Z} & \mathbb{Z} & 0 & 0 & 0 \\
\mathbb{Z} \oplus \mathbb{Z} & \mathbb{Z} \oplus \mathbb{Z} \oplus \mathbb{Z} &
\mathbb{Z} \oplus \mathbb{Z} & \mathbb{Z} & \mathbb{Z} \\
\mathbb{Z} & \mathbb{Z} \oplus \mathbb{Z} & \mathbb{Z} \oplus \mathbb{Z} &
\mathbb{Z} \oplus \mathbb{Z} & \mathbb{Z} \oplus \mathbb{Z} \\
\hline
\end{array}
\end{equation*}
To fill this table we need:
\begin{itemize}
\item the K\"{u}nneth formula for $\mathbb{P}^2 \times
\mathbb{P}^2$, 
\item the fact that the Schubert classes form a basis of the integral
cohomology of the Grassmannian (see \cite{manivel-book}) for $\mathrm{Gr}(2,6)$,
\item the beginning of section $3$ of \cite{maniliev-chow} for
$\mathbb{OP}^2$.
\end{itemize}
Assume that $\AP$ is embedded in $\mathbb{P}^{2 \ma +1}$ as a hypersurface.
Then, by Lefschetz hyperplane theorem, we have:
\begin{equation*}
H^{2k}(\AP, \mathbb{Z}) = H^{2k}(\mathbb{P}^{2 \ma +1},\mathbb{Z}) =
\mathbb{Z},
\end{equation*}
for all $0 \leq k \leq \ma -1$. The above array shows that it is
impossible.

\newpage

\section{$\mathbb{Q}$-factoriality and resolution of singularities}

\subsection{Statement of the result and proof}

Let $X$ be a singular variety. Experience tells us that it is often possible to
decide if $X$ is normal, Cohen-Macaulay, Gorenstein, with rational singularities
(see \cite{weyman-listeI} and
\cite{weyman-listeII} for some lists about prehomogeneous vector spaces).
However
$\mathbb{Q}$-factoriality seems to be much harder to prove.
In \cite{weyman-listeI} and \cite{weyman-listeII}, $\mathbb{Q}$-factoriality is
never discussed. In this appendix, we prove a criterion for
$\mathbb{Q}$-factoriality and we apply it to concrete situations. This
result was implicitly used in the proof of lemma
\ref{formula-canonicalII}. Notice that our result is similar to lemma
$1.1.1$ in \cite{namikawa}.

\begin{prop} \label{Q-fact}
Let $\pi : Y \rightarrow X$ be a birational morphism such that $X$ and $Y$
have Gorenstein rational singularities. Assume that the scheme
theoretic exceptional locus of
$\pi$ (denoted by $E$) is a Cartier divisor in $Y$ such that:
\begin{itemize}
  \item (i) all the fibers of $\pi$ have Picard rank equal to $1$, 
\item (ii) $\w_Y$ is relatively anti-ample with respect to $\pi$,
\item (iii) $\pi(E)$ is irreducible.
\end{itemize}
Then we have:
\begin{equation*}
Y \, \text{is $\mathbb{Q}$-factorial} \, \Rightarrow X \, \text{is
$\mathbb{Q}$-factorial}.
\end{equation*}
\end{prop}

Notice that we do not impose the fibers of $\pi$ to be reduced. One
also easily checks that if $X$ has Gorenstein terminal singularities and:
\begin{equation*}
 X_n \rightarrow X_{n-1} \rightarrow \cdots \rightarrow X_0 = X
\end{equation*}
is a resolution of singularities where all $ \pi_i : X_i \rightarrow X_{i-1}$
are blow-ups along smooth normally flat centers, then all the $\pi_i$'s satisfy
the hypotheses $(ii)$ and $(iii)$ of proposition \ref{Q-fact}. Hence, to apply
this proposition to such a resolution of singularities, the only non-trivial
hypothesis to check is the condition $(i)$. Moreover, this condition on the
Picard rank is sharp, as shown by the following example.

\begin{exem}
Let $V$ be a vector space of dimension $n \geq 2$ and let 
\begin{equation*}
X = \{ A \in \mathrm{End}(V)\,\, \text{such that}\,\, \mathrm{rk}(A) \leq 1 \}.
\end{equation*}
This is a rational singularities Gorenstein variety which is only singular in
$0_{\mathrm{End}(V)}$ (see \cite{weyman}, corollary $6.1.5$). Consider the
incidence:
\begin{equation*}
 \tilde{X} = \{ (A,L,M) \in X \times \mathbb{P}(V) \times \mathbb{P}(V^*)\,\,
\text{such that}\,\, \mathrm{Im}(A) \subset L \,\,\text{and} \,\, M \subset
\mathrm{Ker}(A) \}.
\end{equation*}
The natural projection $\pi : \tilde{X} \rightarrow X$ is a resolution of
singularities and the exceptional locus of $\pi$ (denoted by $E$) is isomorphic
to $\mathbb{P}(V) \times \mathbb{P}(V^*)$ : this is a Cartier divisor in
$\tilde{X}$. One easily shows that $\w_{\tilde{X}} = \pi^* \w_{X} \ot
\OO_{\tilde{X}}((n-1)E)$, so that the condition $(ii)$ and $(iii)$ of
proposition \ref{Q-fact} are satisfied. However $X$ is not
$\mathbb{Q}$-factorial. Indeed, let
\begin{equation*}
 X' = \{ (A,L) \in X \times \mathbb{P}(V) \,\,
\text{such that}\,\, \mathrm{Im}(A) \subset L \}.
\end{equation*}
The projection $p : X' \rightarrow X$ is a resolution of singularities whose
exceptional locus is isomorphic to $\mathbb{P}(V)$ : it has codimension bigger
than $2$ in $X'$. As a consequence of \cite{debarre}, $1.40$, the variety $X$ is
not $\mathbb{Q}$-factorial. 
\end{exem}

\begin{proof}[of proposition \ref{Q-fact}]
We first demonstrate that 
\begin{equation*}
\dim NE(\pi) = 1.
\end{equation*}
Let $x_0 \in \pi(E)$ be a general point. We will show that for all $x \in
\pi(E)$, there
exists two curves $C_0 \subset \pi^{-1}(x_0)$ and $C \subset \pi^{-1}(x)$ such
that $C_0$ and $C$ are numerically equivalent.

Let $S$ be a curve in $X$ passing through $x_0$ and $x$ with $S \subset \pi(E)$.
Let $S' \rightarrow
S$ be the normalization of $S$ and $p : S' \rightarrow X$ the induced morphism.
Let us consider the fiber product:

\begin{equation*}
\xymatrix{
Y' = Y \times_S S' \ar[rr]^{f_{S'}} \ar[dd]^{p'} & & S' \ar[dd]^{p} \\
& & \\
Y \ar[rr]^{f} & & X} 
\end{equation*}
Let $d$ be the dimension of $f^{-1}(x_0)$. Let $Z \subset Y'$, the vanishing
locus of $d-1$ general sections of $\OO_{Y'/S'}(m)$ for $m>>0$ (where
$\OO_{Y'/S}(1)$ is a
relatively ample bundle for $f_{S'}$). Denote by $f_Z : Z \rightarrow S'$ the
morphism obtained by restriction of $f_{S'}$. Then ${f_Z}^{-1}(x_0)$ is
of dimension
$1$. Let $Z_0$ be the reduced space underlying an irreducible component of $Z$
which dominates $S'$. As $S'$ is smooth and
$Z_0$ integral, we conclude that the restriction $f_{Z_0} : Z_0 \rightarrow
S'$ is flat. The fiber $f_{Z_0}^{-1}(x)$ is then of dimension
$1$. Finally, the morphism $p' : Y' \rightarrow Y$ is finite over its image, so
that $p'(f_{Z_0}^{-1}(x))$ and $p'(f_{Z_0}^{-1}(x_0))$ are two
curves numerically equivalent (in $Y$) which are respectively included in
$\pi^{-1}(x)$ and $\pi^{-1}(x_0)$.

\bigskip

We can now prove that $\dim NE(\pi) = 1$. It is sufficient to prove that if
$C$ and $C_0$ are two curves in $\pi^{-1}(x)$
and $\pi^{-1}(x_0)$, then $C$ and $C_0$ are numerically proportional in $Y$. But
the Picard rank of $\pi^{-1}(x)$ is
$1$, so that all curves in $\pi^{-1}(x)$ are numerically proportional to each
other. The same holds for $\pi^{-1}(x_0)$. Since we know
that there exists two curves included in $\pi^{-1}(x)$ and $\pi^{-1}(x_0)$
which are numerically equivalent, we deduce that all curves in $\pi^{-1}(x)$
are numerically proportional to any curve in $\pi^{-1}(x_0)$. Hence, the curves
$C$ and $C_0$ are numerically proportionnal.

\bigskip

Let $x \in \pi(E)$ and $C$ a curve in $\pi^{-1}(x)$. The condition
$(ii)$ of proposition \ref{Q-fact} implies:
\begin{equation*}
 \w_{X}|_{\pi^{-1}(x)}.C <0.
\end{equation*}
Thus, the class $C$ generates a negative ray in $NE(\pi)$. Since $\dim
NE(\pi) = 1$, this is an extremal negative ray. As a consequence, we apply the
relative cone theorem (see theorem $7.51$ of \cite{debarre}) and we find a
diagram:
\begin{equation*}
 \xymatrix{Y \ar[rr]^{c_R} \ar[ddrr]^{\pi} & & Z \ar[dd]^q \\
 & & \\
 & & X}
\end{equation*}
where $c_R$ is the contraction for the negative extremal ray $R =
\mathbb{R}^{+}.[C]$.

\bigskip

Let $x \in \pi(E)$. Assume that $\dim c_R(\pi^{-1}(x)) > 0$. Then, we can find
two curves $C \subset c_R(E)$ and $C' \subset \pi^{-1}(x)$ such that $c_R(C') =
C$. But $\dim NE(\pi) = 1$, so that all curves included in $\pi^{-1}(x)$ are
contracted by $c_R$ (see theorem $7.51$ of
\cite{debarre}), this is a contradiction. We conclude that $q$ is a
finite morphism such that $\RR q_* \OO_Z = \OO_X$. But $X$ is normal, so that
by Zariski's main theorem, the morphsim $q$ is
an isomorphism. As a consequence, the morphism $\pi$ is a divisorial contraction
of a negative extremal ray. As $Y$ is $\mathbb{Q}$-factorial, proposition $7.44$
of \cite{debarre} ensures that $X$ is also $\mathbb{Q}$-factorial.

\end{proof}
\subsection{Applications}
We apply proposition \ref{Q-fact} to some examples.
\begin{cor}
All Pfaffians varieties are $\mathbb{Q}$-factorial.
\end{cor}
This result is certainly well-known to experts, but the only (implicit)
reference I have been able to find is lemma $1.1.1$ in \cite{namikawa}.

\begin{proof}
Let $V$ be a vector space of dimension $n \geq 2$ and $p$ an integer such that
$2p<n$.
We denote by:
\begin{equation*}
Z^{(p)} = \mathbb{P} \{ A \in \bigwedge^2 V, \mathrm{rg} A \leq 2p \}
\end{equation*}
the Pfaffian variety of rank $2p$ in $\mathbb{P}(\bigwedge^2 V)$. Let us
consider the resolution of singularities $\pi : \tilde{Z}^{(p)} \rightarrow
Z^{(p)}$ where:
\begin{equation*}
\tilde{Z}^{(p)} = \{(A,M) \in Z^{(p)} \times \mathrm{Gr}(2p,V), \, \text{such
that} \, Im(A) \subset M \}.
\end{equation*}
The varietry $\tilde{Z}^{(p)}$ is the total space of a projective bundle over
$ \mathrm{Gr}(2p,V)$ : it is irreducible. Hence $Z^{(p)}$ is also irreducible
for any $p$. Let $E$ be the exceptional locus of $\pi : \tilde{Z}^{(p)}
\rightarrow Z^{(p)}$. This is an integral Cartier divisor, and we have $\pi(E)
= Z^{(p-1)}$. Hence, the condition $(iii)$ of \ref{Q-fact} is satisfied.

We show that $Z^{(p)}$ has terminal singularities. As
${Z}^{(p)}$ is Gorenstein
(see \cite{weyman}, proposition $6.4.3$), there is an integer $m \in
\mathbb{Z}$ such that $\w_{\tilde{Z}^{(p)}} = \pi^* \w_{Z^{(p)}} \ot
\OO_{\tilde{Z}^{(p)}}(mE)$ (see definition $2.22$ of \cite{mori-kollar}).
By the adjunction formula we have:
\begin{equation*}
\w_{E} = \pi^* \w_{Z^{(p)}} \ot
\OO_{E}((m+1)E).
\end{equation*}
But the map $\pi : E \rightarrow \pi(E)$ is generically flat, so that the
adjunction formula implies:
\begin{equation*}
 \w_{\pi^{-1}(x)} = \OO_{E}((m+1)E)|_{\pi^{-1}(x)},
\end{equation*}
for generic $x \in \pi(E)$. Moreover, we know that
$\OO_{E}(E)|_{\pi^{-1}(x)}
= \OO_{\pi^{-1}(x)}(-1)$ and that the generic fiber of $\pi : E \rightarrow
\pi(E)$ is isomorphic to $\mathrm{Gr}(2,\mathbb{C}^{n-2p+2})$. Thus, for generic
$x \in
\pi(E)$, we have $\w_{\pi^{-1}(x)} =
\OO_{E}((n-2p+2)E)|_{\pi^{-1}(x)}$.
As $2p < n$, we deduce that $m>1$. The conditions
$(ii)$ \ref{Q-fact} is satisfied. 

Finally, for all $A \in Z^{(p)}$, the fiber of $\pi$ over $A$ is isomorphic to
$\mathrm{Gr}(2p-\mathrm{rg}(A), V/ \mathrm{Im}(A))$ : it has Picard rank
$1$. As a consequence, the condition $(i)$ of \ref{Q-fact} is also satisfied
for the morphism $\pi : \tilde{Z}^{(p)} \rightarrow Z^{(p)}$ and we get that
$Z^{(p)}$ is $\mathbb{Q}$-factorial.

\end{proof}

Another corollary of proposition \ref{Q-fact} is the following:
\begin{cor}
The hypersurface $\tg$ is $\mathbb{Q}$-factorial.
\end{cor}
when $\AA = \mathbb{C}, \mathbb{H}$ or $\mathbb{O}$,  the hypersurface
$\tg$ is smooth in codimension $3$. Grothendieck's factoriality
theorem (see \cite{SGA2}, Expos\'e XI, corollary $3.14$) shows that $\tg$ is
factorial. When $\AA = \mathbb{R}$, the singular locus of $\tg$ has precisely
codimension $3$ in $\tg$. So we can not apply Grothendieck's result.

\begin{proof}
By theorem \ref{desingtau}, we have a sequence of blow-ups:
\begin{equation*}
X_2 \stackrel{\pi_2}\rightarrow X_1 \stackrel{\pi_1}\rightarrow X = \tg
\end{equation*}
such that $X_2$ is smooth and the morphisms $\pi_1$ and $\pi_2$ satisfy
the items $(i)$ and $(iii)$ of proposition \ref{Q-fact}. Moreover, by lemma
\ref{formula-canonicalII}, these morphisms also satisfy the condition $(ii)$ of
\ref{Q-fact}. As a consequence, we can apply proposition \ref{Q-fact} to
$\pi_1$ and $\pi_2$. We deduce that $\tg$ is
$\mathbb{Q}$-factorial.
\end{proof}

\end{appendix}

\newpage

\bibliographystyle{alpha}

\bibliography{biblicrepant}

\end{document}